\documentclass[12pt,reqno]{amsart}

\usepackage{amssymb}
\usepackage{amsthm}

\usepackage[margin=1in]{geometry}
\usepackage{color,graphicx}
\usepackage{thmtools} 
\usepackage[hypertexnames=false]{hyperref}
\usepackage{cleveref}

\usepackage{tikz}
\usetikzlibrary{calc}
\usetikzlibrary{arrows.meta, decorations.markings}
\usepackage{ifthen}
\tikzset{
    network graph/.style={
        x=0.9cm,
        y=1.15cm
    },
    vertex/.style={
        circle,
        draw,
        fill=white,
        minimum size=4.2mm,
        inner sep=0pt,
        font=\scriptsize
    },
    edgelabel/.style={
        font=\scriptsize,
        fill=white,
        inner sep=1.2pt
    },
    directed/.style={
        line width=0.7pt,
        postaction={
            decorate,
            decoration={
                markings,
                mark=at position 0.56 with {
                    \arrow{Stealth[length=1.6mm,width=1.6mm]}
                }
            }
        }
    }
}

\newcommand\nodcnt{\varphi}
\newcommand\R{\mathbb{R}}
\newcommand\C{\mathbb{C}}
\newcommand\Z{\mathbb{Z}}

\newcommand{\E}{{\vec{E}}}

\newcommand{\cP}{\mathcal{P}}
\newcommand{\cH}{\mathcal{H}}
\newcommand{\cS}{\mathcal{S}}
\newcommand{\cA}{\mathcal{A}}

\DeclareMathOperator\Ran{Ran}
\DeclareMathOperator\Hess{Hess}
\DeclareMathOperator\Real{Re}

\DeclareMathOperator\diag{diag}

\newtheorem{theorem}{Theorem}[section]	
\newtheorem{corollary}[theorem]{Corollary}
\newtheorem{lemma}[theorem]{Lemma}
\newtheorem{proposition}[theorem]{Proposition}

\theoremstyle{definition}
\newtheorem{definition}[theorem]{Definition}

\theoremstyle{remark}
\newtheorem{remark}[theorem]{Remark}
\newtheorem{example}[theorem]{Example}

\numberwithin{equation}{section}

\begin{document}

\title{Cycle intersection form and oscillation of graph eigenfunctions}

\author[G.\ Berkolaiko]{Gregory Berkolaiko}
\address{Department of Mathematics,
	Texas A\&M University, College Station,
	TX 77843, USA}
\email{gberkolaiko@tamu.edu}

\author[J.C.\ Bronski]{Jared C. Bronski}
\address{Department of Mathematics, University of Illinois, Urbana, IL 61801, USA}
\email{bronski@illinois.edu}

\author[M.\ Goresky]{Mark Goresky}
\address{Institute for Advanced Study, Princeton, NJ 08540, USA}
\email{goresky@ias.edu}

\begin{abstract}
  For a real symmetric matrix $H$ strictly supported on a finite
  simple graph, it is shown that the inertia of a weighted
  intersection form on the cycle space of the graph, with weights
  derived from a non-vanishing eigenvector of $H$, governs oscillation
  data on the graph.  Specifically, the null space of the form
  controls eigenvalue multiplicity, while its Morse index determines
  the number of sign changes across edges.  In the case of a simple
  eigenvalue, the Hessian at zero of the eigenvalue branch of the
  discrete magnetic Schr\"odinger operator is identified with the dual
  of the cycle intersection form.  Applications are given to stability
  analysis of coupled oscillator networks, to the local behavior of
  dispersion relations for (decorated) strained graphene, and to
  nodal-domain counts.
\end{abstract}

\maketitle

\section{Introduction and the main results}
\label{sec:intro}

Let \(G=(V,E)\) be a simple graph with finite vertex set \(V\). Each
undirected edge in \(E\) is a two-element subset \(\{u,v\}\subset V\);
the corresponding adjacency relation will be denoted by \(u\sim v\).
Fix once and for all an orientation \(\vec E\) of the undirected edge
set \(E\): for each edge \(\{u,v\}\in E\), exactly one of the ordered
pairs \((u,v)\) and \((v,u)\) belongs to \(\vec E\). Thus
\(|\vec E|=|E|\).

The space of real-valued \(1\)-chains on \(G\) is
\begin{equation}
  \label{eq:1chains_def}
  C_1(G,\R)
  :=
  \left\{
    \sum_{e\in\vec E} y_e e \mid y_e\in\R
  \right\}
  \simeq \R^{\vec E}.
\end{equation}
With the convention \(\partial(u,v)=(v)-(u)\), the \emph{cycle
subspace} \(Z\subset C_1(G,\R)\) consists of \(1\)-chains
with zero boundary.  Equivalently, at every vertex
\(u\in V\), the sum of incoming coefficients equals the sum of outgoing
coefficients:
\begin{equation}
  \label{eq:flow_condition}
  \sum_{v:(v,u)\in\vec E} y_{(v,u)}
  =
  \sum_{w:(u,w)\in\vec E} y_{(u,w)}.
\end{equation}
Since there are no \(2\)-chains, the first homology group
\(H_1(G,\R)\) coincides with \(Z\). Its dimension is the
\emph{first Betti number}
\begin{equation}
  \label{eq:beta_def}
  \beta
  :=
  \dim H_1(G,\R)
  =
  |E|-|V|+c,
\end{equation}
where \(c=c(G)\) is the number of connected components of \(G\).

Given a vector $r = \{r_e\}_{e \in \vec{E}}$ of real edge weights,
define the symmetric bilinear form $\mathcal{R}$,
\begin{equation}
  \label{eq:Rform_def}
  \mathcal{R} : C_1(G,\R) \times C_1(G,\R) \to \R,
  \qquad
  \mathcal{R}(y, y') = \sum_{e \in \vec{E}} r_e y_e y'_e.
\end{equation}
The \emph{weighted cycle intersection form}, denoted $\mathcal{X}_r$,
is the restriction of this bilinear form to the cycle subspace
$Z=H_1(G,\R)$.  Concretely, if $R=\diag(r_e)$ and if $C$ is a \emph{frame} for $Z$, i.e.\ 
any $|\vec{E}|\times\beta$ matrix
whose columns form a basis of $Z$, then the matrix
of $\mathcal X_r$ in this basis is $C^\top R C$.  The cycle intersection 
form $\mathcal{X}_r$ turns out to play a remarkable role in the
spectral analysis of the discrete Schr\"odinger operators on the graph
$G$, with the nullity and the Morse index of $\mathcal{X}_r$
determining the multiplicity and nodal count of an
eigenvector.

Recall that for a real symmetric matrix $M$, the \emph{nullity}
$n_0(M)$ is the dimension of its kernel, and the \emph{Morse index}
$n_-(M)$ is the number of its negative eigenvalues.
For a symmetric bilinear form $\mathcal{B}$ on a vector space $W$, the
\emph{nullity} $n_0(\mathcal{B})$ is the dimension of
$\{w \in W \mid \mathcal{B}(w, w') = 0 \text{ for all } w' \in W\}$,
and the \emph{Morse index} $n_-(\mathcal{B})$ is the maximal dimension
of a subspace on which $\mathcal{B}$ is negative definite. By
Sylvester's Law of Inertia, if $B$ is the matrix representation of
$\mathcal{B}$ in any basis, then $n_0(\mathcal{B}) = n_0(B)$ and
$n_-(\mathcal{B}) = n_-(B)$.

A discrete (generalized) Schr\"odinger operator $H$ on $G$ is an
$n \times n$ real symmetric matrix that is \emph{strictly supported on
  $G$}: for any two distinct vertices $u \neq v$, the entry
$H_{u,v} \neq 0$ if and only if $u \sim v$. The eigenvalues of $H$ are real and ordered
$\lambda_1(H) \leq \lambda_2(H) \leq \ldots \leq \lambda_n(H)$. For an
eigenvalue $\lambda$, which may not be simple, its \emph{upper label}
$k^\uparrow$ is
\begin{align}
  \label{eq:upper_label_def}
  k^\uparrow &:= \max\{k\in \mathbb{N} \colon \lambda_k(H) = \lambda
  \} \\
  \nonumber
  &= n_-(H-\lambda) + n_0(H-\lambda).
\end{align}

Let $\psi$ be an
eigenvector of $H$ corresponding to an eigenvalue $\lambda$, and
assume all entries of $\psi$ are non-zero. We define the \emph{nodal
  count} of $\psi$ as
\begin{equation}
  \label{eq:nodal_count_def}
  \nodcnt(\psi,H) := \# \left\{(u,v)\in \vec{E}\colon \psi_u H_{u,v} \psi_v > 0 \right\}.
\end{equation}
We also refer to the nodal count in \eqref{eq:nodal_count_def} as the
\emph{oscillation} of $\psi$; the term reflects the analogy with the
oscillation of Sturm--Liouville eigenfunctions, made precise in
\Cref{sec:motivation_nodal}.  When $H$ is the discrete Laplace
operator, $H_{u,v} = -1$ and the edges counted in oscillations
$\nodcnt(\psi,H)$ are exactly the edges on which $\psi$ changes sign,
$\psi_u\psi_v < 0$; the definition \eqref{eq:nodal_count_def} is
formulated so as to avoid any assumption on the signs of the entries
of $H$.

\begin{theorem}
  \label{thm:main}
  Let $H$ be strictly supported on a finite, simple graph $G$ with
  $c=c(G)$ connected components, and let $\psi \in \ker(H-\lambda)$ be
  an eigenvector of $H$ with no zero entries. Define the edge weights $r$
  in the weighted cycle intersection form $\mathcal{X}_r$ by
  \begin{equation}
    \label{eq:the_r_formula}
    r_e = -(\psi_u H_{u,v} \psi_v)^{-1},
  \end{equation}
  for each edge $e=(u,v) \in \vec{E}$. Then the \emph{multiplicity} of
  the eigenvalue $\lambda$ is
  \begin{equation}
    \label{eq:main_formula_nullity}
    \dim \ker (H-\lambda) = c + n_0(\mathcal{X}_r),
  \end{equation}
  and the \emph{nodal count} of $\psi$ is
  \begin{equation}
    \label{eq:main_formula_surplus}
    \nodcnt(\psi,H) = k^\uparrow - c + n_-\left( \mathcal{X}_r \right).
  \end{equation}
\end{theorem}

\begin{remark}
  \label{rem:Xcoord}
  In terms of the coordinate representation $C^\top R C$ of the form $\mathcal{X}_r$ defined above, 
  the Morse index in \eqref{eq:main_formula_surplus} 
  is $n_-(C^\top R C)$ and the nullity \eqref{eq:main_formula_nullity} is $n_0(C^\top R C)$, both independent 
  of the choice of the frame $C$.
\end{remark}

\begin{remark}
  \label{rem:nonzero}
  \Cref{thm:main} can also be applied to eigenvectors with \emph{some
    zero entries} by removing their rows and columns from the matrix
  $H$ (and the graph).  We showcase this idea by estimating the number
  of \emph{nodal domains} of $\psi$ in \Cref{sec:nod_dom}.
\end{remark}

\begin{remark}
  \label{rem:orientation}
  The operator $H$ and the nodal count $\nodcnt(\psi,H)$ are agnostic of
  the orientation of the edges of $G$.  The nullity and the Morse
  index of the form $\mathcal{X}_r$ are also independent of the
  orientation, even though we used a choice of orientation to define
  $\mathcal{X}_r$.
\end{remark}

When the eigenvalue $\lambda$ is simple, Theorem~\ref{thm:main} has a
useful ``dual companion''.  Since $\dim \ker (H-\lambda) = 1$, the
cycle intersection form $\mathcal{X}_r$ is non-degenerate and
therefore has a \emph{dual} $\mathcal{X}_r^\sharp$ (the definition is
reviewed in \Cref{sec:musical_Hodge}), which turns out to coincide
with the Hessian of the eigenvalue function of the discrete magnetic
Schrödinger operator, an object governing the local behavior of band
functions on abelian covers of the graph $G$ (see
\Cref{sec:motivation_dispersion}).

The discrete magnetic Schrödinger operator is
defined as a \emph{phase-perturbed} matrix:
\[
(H_{\alpha})_{u,v} := e^{i\alpha_{u,v}} H_{u,v}, 
\]
where $\alpha$ is a real valued antisymmetric function of pairs $u,v$
such that $u\sim v$.  The space of such $\alpha$ will be identified
with the space of 1-cochains $C^1(G, \R) := \R^{\vec{E}}$, i.e.\ the
functions from $\vec{E}$ to $\R$; equivalently, with the antisymmetric
matrices $\cA(G)$ (see \Cref{sec:phase_pert}).
The matrix $H_\alpha$ is unitarily
equivalent to $H_{\alpha+d\theta}$ for any $\theta \in \R^V$, where
$d$ is the coboundary operator,
\begin{equation}
  \label{eq:d_def}
  d: \R^V \to C^1(G,\R),
  \qquad
  (d\theta)_{u,v} = \theta_v - \theta_u,
  \quad
  u\sim v.
\end{equation}
Since unitarily equivalent matrices have identical spectra, the
eigenvalue function $\Lambda(\alpha) := \lambda_k(H_\alpha)$ is
constant on affine subspaces $\alpha+\Ran d\subset C^1(G,\R)$.
Therefore, the Hessian $\Hess\Lambda(0)$ vanishes whenever one
argument lies in $\Ran d$ and thus descends to a well-defined bilinear
form on the cohomology $H^1(G,\R)=C^1(G,\R)/\Ran d$.  Using the
natural pairing, $H^1(G, \R)$ can be viewed as the dual space to the
cycle space $Z=H_1(G, \R)$.

\begin{theorem}[Magnetic Hessian]
  \label{thm:magHessian}
  Let $\lambda_k$ be a simple eigenvalue of $H$ with eigenvector
  $\psi$ that has no zero entries. Then the function
  $\Lambda(\alpha) = \lambda_k(H_\alpha)$ has a critical point at
  $\alpha=0$. Its Hessian at this point,
  $\Hess\Lambda\colon C^1(G,\R) \times C^1(G,\R) \to \R$, is a
  bilinear form on $C^1(G,\R)$ given by the dual form
  $\mathcal{X}_r^\sharp$ via
  \begin{equation}
    \label{eq:Hess}
    \frac12\Hess\Lambda(\alpha_1, \alpha_2) =
    \mathcal{X}_r^\sharp\Big([\alpha_1], [\alpha_2]\Big),
    \qquad
    r_e = -(\psi_u H_{u,v} \psi_v)^{-1}.
  \end{equation}
  where $[\alpha] \in H^1(G,\R)$ denotes the equivalence class of an
  $\alpha\in C^1(G,\R)$.

  In coordinates, if $C$ is a frame for the cycle space $Z$, the Hessian is given by
  the matrix $C (C^\top R C)^{-1} C^\top$, i.e.
  \begin{equation}
    \label{eq:Hess_coord}
    \frac12\Hess\Lambda(\alpha_1, \alpha_2) =
    \left<\alpha_1, C (C^\top R C)^{-1} C^\top
      \alpha_2\right>_{\R^{\vec{E}}},
  \end{equation}
  where $R$ is the diagonal matrix of the weights
  $r_e$ and the inner product is the
  standard inner product on $C^1(G,\R) = \R^{\vec{E}}$.
\end{theorem}

\begin{remark}
  \label{rem:dualX_coords}
  We highlight two beautiful features of the coordinates
  representation of the Hessian.  First, it is independent of the
  choice of the basis for the cycle space: replacing $C$ with $CT$,
  where $T$ is $\beta\times\beta$ invertible, leaves it invariant.
  Second, since (see \Cref{sec:coboundary})
  \begin{equation}
    \label{eq:equiv_criterion}
    [\alpha_1] = [\alpha_2] \quad
    \Longleftrightarrow \quad
    C^\top \alpha_1 = C^\top \alpha_2,
  \end{equation}
  the value of the Hessian in \eqref{eq:Hess_coord} clearly depends
  only on the cohomology class of $\alpha_{1,2}$, agreeing with
  \eqref{eq:Hess}.
\end{remark}

\begin{remark}
  \label{rem:fluxHessian}
  In applications such as spectral analysis of discrete periodic
  operators, see \Cref{sec:motivation_dispersion}, one does not
  consider the whole space of phases $\alpha\in C^1(G,\R)$, focusing
  instead on the dependence of $\Lambda$ on the cohomology class
  $[\alpha] \in C^1(G,\R) / \Ran d$.  This is done by restricting
  $\Lambda$ to a ``representative'' $\beta$-dimensional linear
  subspace of $C^1(G,\R)$ which is complementary to $\Ran d$.  This
  subspace is parametrized by the \emph{flux coordinates}
  $C^\top \alpha =: \vec{k} \in \R^\beta$.  In terms of these
  coordinates the Hessian is
  \begin{equation}
    \label{eq:Hess_flux_coords}
    \frac12\Hess\Lambda(\vec{k}_1, \vec{k}_2) =
    \left<\vec{k}_1,  (C^\top R C)^{-1} \vec{k}_2\right>_{\R^\beta},
  \end{equation}
  as seen immediately from \eqref{eq:Hess_coord}.
\end{remark}

\subsection*{Outline}

The paper is structured as follows.  \Cref{sec:motivation} describes
our motivation, placing our results in the context of oscillation
theory on graphs, discussing the role of the magnetic Hessian in
Floquet--Bloch theory for periodic discrete operators, and relating
$\det(C^\top R C)$ to cographic matroids and associated
basis-generating polynomials.

\Cref{sec:examples} illustrates the weighted cycle intersection form
$\mathcal X_r$ through explicit examples and computations,
illustrating how to build frames $C$ for the cycle space and how the
inertia of $C^\top R C$ encodes nodal and multiplicity information.

Sections \ref{sec:aux} and \ref{sec:magnetic} contain the proofs of
\Cref{thm:main} and \Cref{thm:magHessian} correspondingly.  We chose
to express the first proof in the linear-algebraic language, while the
second proof treats its objects as forms rather than matrices.

We conclude with three sections of applications.  \Cref{sec:graphene}
applies the magnetic-Hessian formalism to band functions of
(decorated) strained graphene, where the inverse matrix
$(C^\top R C)^{-1}$ governs the local second-order behavior of the
dispersion relation near symmetry-enforced critical points (the
relevant symmetry, and the resulting criticality, are those of
\eqref{eq:real_symm}).
\Cref{sec:oscillator_networks} applies \Cref{thm:main} to stability
problems for coupled oscillator networks, where the position of a
known eigenvector in the spectrum governs the inertia of a
linearization; we illustrate this in the setting of a non-homogeneous
Kuramoto model on a graph.  Finally, \Cref{sec:nod_dom} applies the
main theorem to the counting of nodal domains on graphs, which is a
question adjacent to counting sign changes.

\subsection*{Acknowledgements}

The authors are grateful to Graham Cox, Robert Marangell and Selim
Sukhtaiev, who at various times suggested that there may be
interesting synergies between to results of J.C.B. and his
collaborators, and of G.B., M.G. and their collaborators.  This work
is the result of such a synergy.  We are also thankful to Ram Band and
Lior Alon, who on numerous occasions explained to us their
understanding of the CdV gauge (cf.\ \Cref{lem:Hodge_flat} and
\Cref{lem:gauge_useful} below).  The authors are grateful to Graham
Cox for a careful reading of the manuscript, and to Galen
Dorpalen-Barry and Cynthia Vinzant for several pointers to the
literature we cite in \Cref{sec:motivation_matroids}.  Portions of the
manuscript were revised with the assistance of the AI language models
ChatGPT, Gemini, and Claude.

This material is based upon research supported by a grant (G.B.) from
the Institute for Advanced Study, School of Mathematics.  G.B. also
acknowledges support from the NSF Grants DMS-2247473, DMS-2510345 and
from the Association of Former Students of Texas A\&M University
through the Faculty Development Leave program.  J.C.B. acknowledges
the support from the NSF Grant DMS-2407358.

\section{Motivation and related results}
\label{sec:motivation}

\subsection{Nodal counting on graphs}
\label{sec:motivation_nodal}

Theorem~\ref{thm:main} belongs to the broad circle of ideas known as
\emph{oscillation theory}, originating in Sturm's work
\cite{Stu_jmpa36}.  In this theory, one relates the position \(k\) of
an eigenvalue in the spectrum to the zero set, or sign-changing set, of
a corresponding eigenfunction.  Such results have two complementary
uses.  In one direction, knowing \(k\) gives information about the
geometry or topology of the nodal set and its complement, which is
useful in clustering and partitioning problems
\cite{Fie_cmj75b,Spi_ICM10}.  In the other direction, knowing an
eigenvector may allow one to determine whether it is the ground state,
as happens, for example, in the linear stability analysis of solutions
to nonlinear equations \cite{GriShaStr_jfa87}.

An application of the latter type is discussed in
\Cref{sec:oscillator_networks} below, where we apply
\eqref{eq:main_formula_surplus} to the stability analysis of a known
solution of a coupled oscillator network.  In this direction,
\Cref{thm:main} extends earlier results of Bronski, DeVille and
Ferguson \cite{BroDevFer_siamjam16}, who derived
\eqref{eq:main_formula_surplus} under the assumptions that
\(\psi=\mathbf 1\) and that the corresponding eigenvalue is simple.
Our result applies to more general oscillator networks; we illustrate
this with a non-homogeneous Kuramoto model.

A key point of \Cref{thm:main} is that the same finite-dimensional
object, the weighted cycle intersection form \(\mathcal X_r\),
controls two phenomena: eigenvalue multiplicity and nodal count.  To
place this within the context of existing literature, it is well known
that all eigenvalues are simple on a path graph
\cite[Prop.~2.2.5]{DamanikFillman_1dSchroedI}.  Fiedler
\cite[Prop.~1]{Fie_cmj75} showed that, on a tree, if the eigenspace of
\(\lambda\) contains an eigenvector with no zero entries, then
\(\lambda\) is simple.  Formula \eqref{eq:main_formula_nullity}
extends this principle to graphs with cycles: if the eigenspace
contains a vector with no zero entries, then, apart from the
unavoidable contribution of the connected components, eigenvalue
multiplicity can occur only through degeneracies of the cycle
intersection form.  In particular, it is bounded by \(\beta+c(G)\).

Moving to sign changes, Gantmacher and Krein \cite{GanKre_oscillation}
showed that the \(k\)-th eigenvector \(\psi_k\) of a path graph
changes sign exactly \(k-1\) times, in direct analogy with Sturm's
theorem for second-order differential operators on an interval.
Fiedler \cite{Fie_cmj75} proved that the same ``Sturm count'' holds on
tree graphs\footnote{And, remarkably, \emph{only} on tree graphs, as
  shown almost 40 years later by Band \cite{Ban_ptrsa14} using the
  nodal-magnetic theorem, \Cref{cor:nodal_magnetic}.} provided
the eigenvalue is simple and the eigenvector has no zero entries.
Berkolaiko \cite{Ber_cmp08} extended this to arbitrary connected graphs by
proving the bound
\begin{equation}
  \label{eq:nodal_bound}
  0 \leq \nodcnt(\psi,H) - (k-1) \leq \beta,
\end{equation}
where \(\beta\) is the first Betti number of the graph.  Inequality
\eqref{eq:nodal_bound} suggests that the \emph{nodal surplus}
$\nodcnt(\psi,H)-(k-1)$ is ``created'' by the cycles of the graph.
Theorem~\ref{thm:main} makes this precise: the surplus is the Morse
index of the weighted cycle intersection form \(\mathcal X_r\).  In
\Cref{sec:nod_dom} we apply this theorem to provide short proofs for
some lower bounds for a related quantity, the number of nodal domains
\cite{Ber_cmp08,Biy_laa03,DavGlaLeySta_laa01,DeiPutTud_acha23,GeLiu_cvpde23,Leydold_nodal,XuYau_jc12}.

This interpretation also clarifies several results and conjectures
about the nodal count and nodal domain count statistics, a
subject of substantial recent interest
\cite{AloBanBer_cmp17,AloGor_prep24,AloMikUrs_prep24,BanOreSmi_incoll08,BerRazSmi_jpa12,GanMckMohSri_cmp23,MckUrs_imrn24}.
For graphs with disjoint cycles, \(\nodcnt(\psi,H)\) has been shown to
follow a binomial distribution \cite{AloBanBer_cmp17,AloGor_prep24}.
Theorem~\ref{thm:main} recasts these statistical questions as
questions about the inertia of the weighted cycle intersection form
\(\mathcal X_r\).  For graphs with disjoint cycles, this form is
diagonal in a natural cycle basis, so each cycle contributes either
\(0\) or \(1\) to the Morse index in \eqref{eq:main_formula_surplus};
this is the intuition behind the binomial law.  More generally, if the
graph has several blocks, meaning maximal biconnected subgraphs, then
\(\mathcal X_r\) has the corresponding block structure.  This suggests
that graphs with many small blocks should exhibit approximately
Gaussian nodal-count statistics.  Numerical evidence further suggests
that \(\nodcnt(\psi,H)\) is Gaussian for broad classes of large
graphs; see \cite{AloBanBer_em22,AloBerGor_prep25} for precise
formulations of this conjecture.

Finally, combining \Cref{thm:main} with \Cref{thm:magHessian} we
immediately recover the nodal--magnetic theorem of Berkolaiko
\cite{Ber_apde13} and Colin de Verdi\`ere \cite{Col_apde13}.

\begin{corollary}[Nodal--magnetic theorem \cite{Ber_apde13,Col_apde13}]
  \label{cor:nodal_magnetic}
  Let \(\lambda_k\) be a simple eigenvalue of \(H\) strictly supported
  on a connected graph, with eigenvector \(\psi\) having no zero
  entries. Then the Morse index of the function
  \(\Lambda(\alpha)=\lambda_k(H_\alpha)\) at the critical point
  \(\alpha=0\) is
  \begin{equation}
    \label{eq:Morse_index}
    n_-\big(\Hess\Lambda\big)
    =
    \nodcnt(\psi,H)-(k-1).
  \end{equation}
\end{corollary}

Indeed, by \Cref{thm:magHessian}, the Hessian of \(\Lambda\) is the
dual form to \(\mathcal X_r\), and therefore has the same Morse index as
\(\mathcal X_r\).  The identity \eqref{eq:Morse_index} then follows
from \eqref{eq:main_formula_surplus} of \Cref{thm:main}.

Our proofs of \Cref{thm:main} and \Cref{thm:magHessian} combine ideas
from Colin de Verdi\`ere's proof \cite{Col_apde13}, see also the
recent exposition in
\cite[Sec.~3.5]{NaderiPankrashkin_SpectralGraphTheory}, with the
approach of Bronski, DeVille and Ferguson
\cite{BroDevFer_siamjam16}.  The resulting picture is that the cycle
space carries an all-important weighted form \(\mathcal X_r\): its
nullity accounts for eigenvalue multiplicity, its Morse index accounts
for nodal surplus, and, in the simple-eigenvalue case, its dual is the
magnetic Hessian.  The last point is particularly useful in the
spectral analysis of periodic discrete operators, discussed next.

\subsection{Periodic discrete operators}
\label{sec:motivation_dispersion}

The magnetic Hessian formula of \Cref{thm:magHessian} has a natural
interpretation in the spectral theory of periodic discrete operators.
We now briefly recall the Floquet--Bloch theory of periodic discrete
operators and its relation to magnetic perturbations; further
references include \cite[Chap.~4]{BerKuc_graphs} and
\cite{KorSab_jmaa14,KotSun_incol03,Kuc_bams16,ShiSot_ot26}.

A graph \(\widetilde G=(\widetilde V,\widetilde E)\) is called
\(\mathbb Z^d\)-periodic if \(\mathbb Z^d\) acts on \(\widetilde G\)
by graph isomorphisms, freely and with finite quotient.  The
\emph{quotient graph} $G=\widetilde G/\mathbb Z^d$ has, as vertices
and edges, the \(\mathbb Z^d\)-orbits of vertices and edges of
\(\widetilde G\), with incidence inherited from \(\widetilde G\) (see
\Cref{fig:graphene} in \Cref{sec:graphene} for an example). Then,
\(\widetilde G\) is an \emph{abelian cover} of the finite graph \(G\),
obtained by periodically repeating \(G\).

A self-adjoint operator \(\widetilde H\) on \(\ell^2(\widetilde V)\),
supported on \(\widetilde G\), is called periodic if \(\widetilde H\)
commutes with the \(\mathbb Z^d\)-action.  Namely, matrix coefficients of $\widetilde H$ are
invariant under simultaneous translation of both arguments:
\[
  \widetilde H_{\tilde u+m,\tilde v+m}
  =
  \widetilde H_{\tilde u,\tilde v},
  \qquad
  \tilde u,\tilde v\in\widetilde V,\quad m\in\mathbb Z^d,
\]
where \(\tilde u+m\) denotes the action of the group element \(m\) on
the vertex \(\tilde u\).

Floquet--Bloch theory decomposes such an operator into
finite-dimensional fiber operators over the \emph{Brillouin torus}
$\mathbb T^d := \mathbb R^d/(2\pi\mathbb Z^d)$.  Namely,
\begin{equation}
  \label{eq:IntegralDecomp}
  \widetilde H
  \simeq
  \int_{\mathbb T^d}^{\oplus} H(\vec{k})\,d\vec{k},
\end{equation}
where \(\vec{k}\in\mathbb T^d\) is the \emph{quasimomentum} and
\(H(\vec{k})\) is a finite matrix acting on functions on the quotient
graph \(G\).  We refer to \cite[Vol 4, Thm XIII.97]{ReedSimon_v14} for
the general direct-integral formalism and to the periodic-graph
literature for this construction in the discrete setting; see, for
example,
\cite{HarKucSob_jpa07,HigShi_ymj99,KorSab_jmaa14,Kuc_bams16,ShiSot_ot26,Sunada_TopologicalCrystallography}.
An immediate consequence of \eqref{eq:IntegralDecomp} is that the
spectrum $\sigma(\widetilde H)$ decomposes into bands,
\begin{equation}
  \label{eq:SpectrumDecomp}
  \sigma(\widetilde H)=\bigcup_{j=1}^{|V|} B_j,
  \qquad
  B_j:=\{\lambda_j(H(\vec{k})):\vec{k}\in\mathbb T^d\}.
\end{equation}
The \emph{bands} $B_j$ may overlap; open intervals in the complement
of their union are \emph{spectral gaps}.  The eigenvalue functions
$\vec{k} \mapsto \lambda_j(H(\vec{k}))$ are called the
\emph{dispersion relations}, or \emph{Bloch band functions}, of the
periodic operator.

We now focus on the case relevant to the magnetic Hessian formula, \Cref{thm:magHessian}.
Let \(G\) be a connected finite graph with first Betti number
\(\beta\).  There is a distinguished \(\mathbb Z^\beta\)-periodic
cover of \(G\), called the \emph{maximal abelian cover}.  It is unique
up to isomorphism and is maximal in the sense that every connected
abelian cover of \(G\) is obtained as its quotient, see \cite[Thm
6.1]{Sunada_TopologicalCrystallography}.  In the rest of this
section, assume that \(G\) is simple and that \(\widetilde G\) is
this maximal abelian cover.

Choose a spanning tree \(\mathcal T\subset G\).  The complement
\(E(G) \setminus E(\mathcal T)\) consists of \(\beta\) edges; fix an
orientation and label them $e_1,\ldots,e_\beta$.  These edges
determine a cycle basis.  In a suitable Floquet gauge, and using the
notation of \Cref{sec:intro}, the fiber operator
in~\eqref{eq:IntegralDecomp}-\eqref{eq:SpectrumDecomp} is obtained as
\[
  H(\vec{k})=H_{\alpha_{\vec{k}}},
  \qquad
  \vec{k}=(k_1,\ldots,k_\beta)\in\mathbb T^\beta
\]
where the \(1\)-cochain \(\alpha_{\vec{k}}\in C^1(G,\mathbb R)\) is
given by
\begin{align*}
  \alpha_{\vec{k}}(e_j)=k_j,\qquad e_j \in E(G) \setminus E(\mathcal T),
  \qquad
  \alpha_{\vec{k}}(e)=0,\qquad e\in E(\mathcal T).
\end{align*}
Thus the quasimomenta \(k_j\) are precisely the magnetic fluxes through
the fundamental cycles.  

In these coordinates, \Cref{thm:magHessian}, or, more precisely,
equation \eqref{eq:Hess_flux_coords}, gives an explicit formula for
the quadratic part of a band function at \(\vec{k}=0\).  Thus, in the
maximal abelian cover, the effective quadratic behavior of the
dispersion relation is given by the inverse of the weighted cycle
intersection matrix.  This observation is trivially extended to other
\emph{corner points} $\vec{k} \in \{0,\pi\}^\beta$ by changing the
sign of the matrix elements $H_{u_j,v_j}$ corresponding to $j$ with
$k_j=\pi$.

Earlier appearances of (unweighted) cycle-intersection forms in this
context addressed the lowest band of a periodic graph or topological
crystal; see \cite[Lem.~2.1]{PhiSar_dmj87} and
\cite[Thm.~6.1]{KotShiSun_jfa98}.  \Cref{thm:magHessian} extends this
picture to arbitrary band edges generated by corner points;
\Cref{sec:graphene} contains an example application to the study of
band functions of graphene.

To conclude, we mention that the local behavior of a band function
near a critical point is central in solid state physics and spectral
theory.  At a non-degenerate band edge, the Hessian
\((C^\top R C)^{-1}\) gives the inverse effective-mass tensor up
to the standard normalization constants; more generally, its signature
determines whether the critical point is an extremum or a saddle.
Such second-order data enter the study of effective Hamiltonians, van
Hove singularities, and dispersive properties of lattice evolution
equations; see \cite{AshcroftMermin_solid,Kuc_bams16,KorSab_jmaa16}.

\subsection{Cycle intersection form, electrical networks, and cographic matroids}
\label{sec:motivation_matroids}

We collect here some references on the meaning of the weighted cycle
intersection form and its determinant in the hope that further research
may reveal interesting corollaries of our results.

The most familiar interpretation of $\mathcal X_r$ comes from
electrical network theory: when \(r_e>0\) is interpreted as the
resistance of edge \(e\), the quadratic form
\begin{equation}
  \label{eq:power_dissipated}
  \mathcal X_r(a,a)=\sum_{e\in\vec E} r_e a_e^2
\end{equation}
is the total power dissipated by a cycle current \(a\), see \cite[Sec
2.4]{LyonsPeres_ProbTreesNetworks}; the matrix $C^\top R C$
representing $\mathcal X_r$ in a cycle basis (\Cref{sec:intro}) is
called the \emph{loop impedance matrix} in circuit analysis
\cite[Sec.~5.2]{Chua_LinearNonlinearCircuits}.  For positive
resistances the form is positive definite on the cycle space, whereas
the spectral problems considered here naturally lead to the signed
weights
as in \eqref{eq:the_r_formula}, and the nodal surplus measures the
number of negative directions of this signed form.

If $T$ is a spanning forest of $G$ (i.e., a spanning tree of each
connected component), the edges $e \in T^*=E \setminus T$ in its
complement form a {\em coforest} (a {\em cotree}, when $G$ is
connected).  The choice of $T$ uniquely defines a basis of the cycle
space $H_1(G;\R)$: to each edge $e \in T^*$ we associate the cycle
that traverses $e$ once in the positive orientation and no other edge
of $T^*$.  This basis defines the $|E|\times \beta$ matrix $C$ of
\Cref{sec:intro}: the entry $C_{e,\gamma}$ is the oriented incidence
number $+1$, $-1$, or $0$ of the edge $e$ in the basis cycle $\gamma$.

Importantly, the basis constructed from $T$ as above is not only a
basis (over the reals) of the linear space $H_1(G;\R)$, but also (over
the integers) of the lattice $H_1(G;\Z)$
\cite{AnBakKupSho_fms14,BacHarNag_bsmf97,Big_blms97,KotSun_aam00,Sjo_mcm91}.  Let $A$
be a set of $\beta$ edges of $G$ and let $C_A$ be the
$\beta \times \beta$ submatrix of $C$, obtained by collecting the rows
indexed by $A$.  Then
\begin{equation}
  \label{eq:coforest_det}
  \det C_A =
  \begin{cases}
    \pm1, & \text{if } A \text{ is a coforest},\\
    0, & \text{otherwise}.
  \end{cases}
\end{equation}
To see this, observe that $C_{T^*}$, where
$T$ is the forest that defined $C$, is the identity matrix.  For any
other coforest $A$, there is a matrix $C'$ obtained from the forest
$T' = E\setminus A$ and thus having $C'_A = I$.  Since the columns of
both $C$ and $C'$ are bases of the lattice $H_1(G;\Z)$, there is a
matrix $S$ with integer entries such that $C' = CS$.  Restricting to
the rows indexed by $A$, we get $\det C_A \det S = \det C'_A = 1$ and,
finally, $\det C_A = \pm1$, because $\det S \in \Z$.  When $A$ is not
a coforest, there is a cycle which contains no edges from $A$ and we
can similarly engineer $C'$ so that $C'_A$ has a column of zeros,
which analogously leads to $\det C_A = 0$.

In fact, more is true: every square submatrix of $C$, of any size, has
determinant $+1$, $-1$, or $0$, i.e.\ the matrix $C$ is {\em totally
  unimodular} \cite[Sec 5.4]{Tutte_matroids}.  This follows from the
case of the maximal ($\beta \times \beta$) minors above because $C$
contains the identity block $C_{T^*}$; see \cite[Chap
19]{Schrijver_LinearAndIntegerProg} for this and further information
on total unimodularity.

In matroid language, the rows of the matrix $C$ form a representation
of Tutte's cographic matroid $M^*(G)$ --- the matroid on edges $E$
whose independent sets are coforests and their subsets
\cite[Sec 5.6]{Tutte_matroids}, \cite[Sec.~2.3]{Oxley_matroid}.  In this context, the {\em cycle
  intersection matrix}
\begin{equation}
  \label{eq:cographic_form}
  \mathcal{X}_r = C^\top R C
  =
  \sum_{e\in\vec E} r_e c_e^{\top} c_e,
\end{equation}
where $R = \operatorname{diag}(r_e)$ is a diagonal matrix of weights
on the graph, is assembled directly from the rows $c_e$ of $C$
representing $M^*(G)$.  As in \cite[Cor.~1]{Ria_dam12} or
\cite[Thm.~2.12]{BroDevFer_siamjam16}, the Cauchy-Binet formula gives
\begin{equation}
  \label{eq:det_cographic_basis}
  \det(C^\top R C)
  =
  \sum_{\substack{A\subset E\\ |A|=\beta}}
  \det(C_A)^2 \prod_{e\in A} r_e
  =
  \sum_{T^*} \prod_{e\in T^*} r_e,
\end{equation}
where the last sum is over coforests $T^*$ of $G$, by virtue of
\eqref{eq:coforest_det}.  In other words, the determinant
$\det(C^\top R C)$ is the {\em basis generating polynomial} of the
cographic matroid $M^*(G)$, also known as the cographic Kirchhoff
polynomial, or first Symanzik polynomial, depending on convention; see
\cite{Tutte_graph_theory,BogWei_ijmpa10}.  It is a {\em stable
  polynomial}, cf.~\cite{BorBra_dmj08,ChoOxlSokWag_aam04}.  At
\(r_e \equiv 1\), \eqref{eq:det_cographic_basis} recovers the number
of spanning forests of \(G\), in agreement with the matrix-tree
theorem.  If \(r_e\) is interpreted as the length of edge \(e\), then
\(\det(C^\top R C)\) gives the squared covolume of the cycle lattice,
equivalently the squared volume of the Jacobian torus of the
corresponding metric graph; see \cite[Thm.~5.2]{AnBakKupSho_fms14}.

From the point of view of the present paper, the significance of
\Cref{thm:main} is that it gives spectral meaning not only to the
determinant \eqref{eq:det_cographic_basis}, but to the full inertia of
the cographic form \eqref{eq:cographic_form}.  As explained in
\Cref{sec:CIF_examples} below, given the graph $G$ and the weights $R$
there exists a discrete Schr\"odinger operator $H$ with an eigenvector
$\psi$ such that the given weights satisfy
\eqref{eq:the_r_formula}. Then \Cref{thm:main} gives
\begin{equation}
  \label{eq:inertia_spectral}
  n_0(\mathcal X_r)=
  \dim\ker(H-\lambda)-c(G), \qquad\text{and}\qquad
    n_-(\mathcal X_r)=
  \nodcnt(\psi,H)-\big(k^\uparrow-c(G)\big),
\end{equation}
where $\nodcnt(\psi,H)$ is the nodal count of $\psi$.

Therefore the degeneracy locus $\det(C^\top R C)=0$, viewed as a
condition on the weights $R=\diag(r_e)$, is precisely the locus where
the corresponding eigenvalue multiplicity exceeds the contribution
forced by the connected components.  Away from
this hypersurface, the nodal surplus is locally constant as a function
of the weight vector \(r\), and its parity is determined by the sign of
the cographic basis polynomial:
\begin{equation}
  \label{eq:surplus_parity_det}
  (-1)^{\nodcnt(\psi,H)-k^\uparrow+c(G)}
  =
  \operatorname{sgn}\det(C^\top R C),
\end{equation}
whenever \(\mathcal X_r\) is nondegenerate. 
Thus the real algebraic stratification of the cographic
basis-generating polynomial controls the possible nodal surplus and
eigenvalue multiplicity of the corresponding Schr\"odinger operator $H$.

\section{Some examples}
\label{sec:examples}

\subsection{Examples of the weighted cycle intersection form}
\label{sec:CIF_examples}

Consider the three oriented graphs shown in \Cref{fig:example_graphs},
listed from left to right as \(G_1,G_2,G_3\).  We write \(e_j\) for
the edge labeled \(j\), and abbreviate $r_j := r_{e_j}$.
The matrices below are written with respect to the depicted
orientations.  Reversing the orientation of an edge changes the
corresponding row of a cycle frame by a sign, but does not change the
inertia of the resulting cycle intersection form.

\begin{figure}
  \centering
\begin{tikzpicture}[network graph,scale=1.3]
\node[vertex] (v1) at (-1.5,4) {$1$};
\node[vertex] (v2) at (0,4) {$2$};
\node[vertex] (v4) at (0,3) {$4$};
\node[vertex] (v3) at (1,3.5) {$3$};
\node[vertex] (v5) at (0,2) {$5$};
\node[vertex] (v7) at (0,1) {$7$};
\node[vertex] (v6) at (1,1.5) {$6$};
\node[vertex] (v8) at (-1.5,1) {$8$};

\draw[directed] (v1) -- node[edgelabel, pos=0.5, above right=3pt] {$1$} (v2);
\draw[directed] (v2) -- node[edgelabel, pos=0.28, left=4pt] {$2$} (v4);
\draw[directed] (v4) -- node[edgelabel, pos=0.30, below right=2pt] {$3$} (v3);
\draw[directed] (v3) -- node[edgelabel, pos=0.30, above right=2pt] {$4$} (v2);
\draw[directed] (v4) -- node[edgelabel, pos=0.28, left=4pt] {$5$} (v5);
\draw[directed] (v5) -- node[edgelabel, pos=0.30, left=4pt] {$6$} (v7);
\draw[directed] (v7) -- node[edgelabel, pos=0.30, below right=3pt] {$7$} (v6);
\draw[directed] (v6) -- node[edgelabel, pos=0.30, above right=3pt] {$8$} (v5);
\draw[directed] (v7) -- node[edgelabel, pos=0.5, above left=3pt] {$9$} (v8);
\end{tikzpicture}
\qquad
\begin{tikzpicture}[network graph]
\node[vertex] (v1) at (4,0) {$1$};
\node[vertex] (v2) at (2,2) {$2$};
\node[vertex] (v3) at (2,-2) {$3$};
\node[vertex] (v4) at (0,0) {$4$};

\draw[directed] (v1) -- node[edgelabel, pos=0.5, above right=2pt] {$1$} (v2);
\draw[directed] (v1) -- node[edgelabel, pos=0.5, below right=2pt] {$2$} (v3);
\draw[directed] (v2) -- node[edgelabel, pos=0.5, right=4pt] {$3$} (v3);
\draw[directed] (v2) -- node[edgelabel, pos=0.5, above left=2pt] {$4$} (v4);
\draw[directed] (v3) -- node[edgelabel, pos=0.5, below left=2pt] {$5$} (v4);

\end{tikzpicture}
\qquad
\begin{tikzpicture}[network graph]
\node[vertex] (v1) at ( 1.5, 0) {$1$};
\node[vertex] (v2) at ( 0,    2) {$2$};
\node[vertex] (v3) at ( 0,   -2) {$3$};
\node[vertex] (v4) at (-1, 0) {$4$};
\node[vertex] (v5) at ( -2, 0) {$5$};

\draw[directed] (v1) -- node[edgelabel, pos=0.5, above right=3pt] {$1$} (v2);
\draw[directed] (v1) -- node[edgelabel, pos=0.5, below right=3pt] {$2$} (v3);
\draw[directed] (v2) -- node[edgelabel, pos=0.5, right=4pt] {$3$} (v3);
\draw[directed] (v2) -- node[edgelabel, pos=0.5, below right=2pt] {$4$} (v4);
\draw[directed] (v3) -- node[edgelabel, pos=0.5, above right=2pt] {$5$} (v4);
\draw[directed] (v2) -- node[edgelabel, pos=0.5, above left=2pt] {$6$} (v5);
\draw[directed] (v5) -- node[edgelabel, pos=0.5, below left=2pt] {$7$} (v3);

\end{tikzpicture}
  \caption{Example graphs with vertex and edge labels and a choice of orientation.}
  \label{fig:example_graphs}
\end{figure}

For the first graph \(G_1\), the cycle space has dimension
\(\beta=2\).  The pendant edges \(e_1,e_9\) and the bridge \(e_5\) do
not occur in any cycle.  A convenient cycle frame is given by
\begin{equation}
  \label{eq:frame_ex1}
  C_1^\top =
  \begin{pmatrix}
    0 & 1 & 1 & 1 & 0 & 0 & 0 & 0 & 0\\
    0 & 0 & 0 & 0 & 0 & 1 & 1 & 1 & 0
  \end{pmatrix}.
\end{equation}
Thus
\begin{equation}
  \label{eq:X_ex1}
  C_1^\top R_1 C_1
  =
  \begin{pmatrix}
    r_2+r_3+r_4 & 0\\
    0 & r_6+r_7+r_8
  \end{pmatrix}.
\end{equation}
The form is diagonal because the two cycles of \(G_1\) are
edge-disjoint.  In particular, each cycle contributes one independent
one-dimensional summand to the inertia of \(\mathcal X_r\).

For the second graph \(G_2\), again \(\beta=2\).  With the edge
numbering and orientations shown in \Cref{fig:example_graphs}, we may
take
\begin{equation}
  \label{eq:frame_ex2}
  C_2^\top =
  \begin{pmatrix}
    1 & -1 & 1 & 0 & 0\\
    0 & 0 & 1 & -1 & 1
  \end{pmatrix}.
\end{equation}
The corresponding matrix of the weighted cycle intersection form is
\begin{equation}
  \label{eq:X_ex2}
  C_2^\top R_2 C_2
  =
  \begin{pmatrix}
    r_1+r_2+r_3 & r_3\\
    r_3 & r_3+r_4+r_5
  \end{pmatrix}.
\end{equation}
Here the two chosen cycles share the edge \(e_3\), and this shared edge
is responsible for the off-diagonal entry \(r_3\).

For the third graph \(G_3\), the cycle space has dimension
\(\beta=3\).  In one choice of cycle frame, the weighted cycle
intersection matrix is
\begin{align}
  \label{eq:X_ex3}
  C_3^\top R_3 C_3
  &=
  \begin{pmatrix}
    r_1+r_2+r_3 & r_3 & r_3\\
    r_3 & r_3+r_4+r_5 & r_3\\
    r_3 & r_3 & r_3+r_6+r_7
  \end{pmatrix} \\ \nonumber
  &=
  \operatorname{diag}(r_1+r_2,r_4+r_5,r_6+r_7)+r_3 \mathbf 1\mathbf 1^\top.
\end{align}
In particular, the cycle intersection form is a rank-1 perturbation of
the diagonal matrix.  This echoes the nodal statistics of such
``watermelon'' graphs, which were shown\footnote{In the context of
  quantum rather than discrete graphs.} in \cite{AloBanBer_em22} to be
a small perturbation of the binomial distribution.

To illustrate the general identity \eqref{eq:det_cographic_basis} for
the determinant, see also \cite[Thm.~2.12]{BroDevFer_siamjam16}, we
compute for \(G_2\),
\begin{equation}
  \label{eq:det_ex2}
  \det(C_2^\top R_2 C_2)
  =
  r_1r_3+r_1r_4+r_1r_5
  +r_2r_3+r_2r_4+r_2r_5
  +r_3r_4+r_3r_5.
\end{equation}
Each monomial here is the product of the weights on the complement of a
spanning tree of \(G_2\).  As a further example, setting all \(r_e=1\) gives
\[
  \det(C_1^\top C_1)=9,\qquad
  \det(C_2^\top C_2)=8,\qquad
  \det(C_3^\top C_3)=20,
\]
the corresponding numbers of spanning trees; this is the Kirchhoff
matrix-tree theorem.

These elementary determinants already encode spectral information.  In
the spectral setting of \Cref{thm:main}, when the weights arise from
an eigenpair with nonzero eigenvector, the zero set of
\(\det(C^\top R C)\) is precisely the locus where the corresponding
eigenvalue is multiple.  Away from this zero set, the sign of the
determinant determines the nodal surplus mod \(2\).

Conversely, one can construct examples with prescribed degeneracies as
follows.  Choose a nonzero vector \(\psi\) and choose nonzero
symmetric off-diagonal entries \(H_{uv}=H_{vu}\) for \(u\sim v\), thus
determining $r$.  Then, for any chosen value of \(\lambda\), define
\[
  H_{uu}
  =
  \lambda-\frac{1}{\psi_u}\sum_{v\sim u}H_{uv}\psi_v .
\]
This ensures that \(H\psi=\lambda\psi\).  Placing the weight
vector \(r\) on the determinantal locus \(\det(C^\top R C)=0\), or on
a higher-codimension locus where \(n_0(C^\top R C)\) is larger,
produces eigenvalues of higher multiplicity.  The reason this
construction is so direct is that the weighted cycle intersection
matrix does not explicitly involve the diagonal entries of \(H\) or
the eigenvalue \(\lambda\), although these are constrained by the
equation \(H\psi=\lambda\psi\).

\subsection{An example of the magnetic Hessian}
\label{ex:mag_Hessian}
  
We now illustrate \Cref{thm:magHessian} with a numerical example.
Consider $H$ strictly supported on the diamond graph $G$, see
\Cref{fig:example_graphs}(middle), given by
\begin{equation}
  \label{eq:H_example}
  H = \begin{pmatrix}
    -1 & 2 & 2 & 0 \\
    2 & 2 & -1 & -1 \\
    2 & -1 & -1 & 4 \\
    0 & -1 & 4 & -3 
  \end{pmatrix}
  \qquad \qquad
  H_\alpha = \begin{pmatrix}
    -1 & 2 e^{i k_1} & 2 & 0 \\
    2 e^{-i k_1} & 2 & -1 & -1 \\
    2 & -1 & -1 & 4 e^{i k_2}\\
    0 & -1 & 4 e^{-i k_2} & -3 
  \end{pmatrix}
\end{equation}
The cycle space of the diamond graph has dimension $\beta=2$ and its
frame can be chosen to be
\begin{equation}
  \label{eq:C_again}
  C^\top= \begin{pmatrix}
    1 & -1 & 1 & 0 & 0 \\
    0 & 0 & 1 & -1 & 1
  \end{pmatrix},
\end{equation}
where the edge numbering and orientation is given in Figure
\ref{fig:example_graphs}.  The
matrices $H_\alpha$ in \eqref{eq:H_example} are parametrized by the cochains
\begin{equation}
  \label{eq:alpha_space}
  \alpha =
  \begin{pmatrix}
    k_1 & 0 & 0 & 0 & k_2
  \end{pmatrix}^\top \in \R^{\vec{E}} = C^1(G,\R),
\end{equation}
where the edges in $\vec{E}$ have been ordered according to their
labels in \Cref{fig:example_graphs}.  The flux coordinates can be
computed by the matrix multiplication,
$C^\top \alpha =
  \begin{pmatrix}
    k_1 & k_2
  \end{pmatrix}^\top
  =: \vec{k}
  \in \R^\beta$.

\begin{figure}
  \centering
  $\nodcnt(\psi_1,H)=1$
  \includegraphics[width=0.3\linewidth]{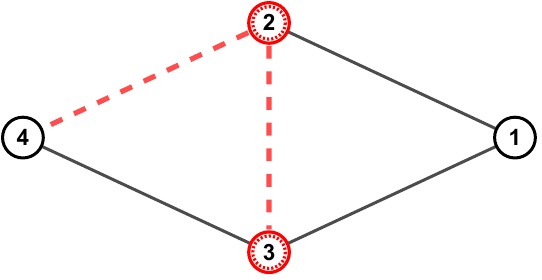}
  \quad
  $\nodcnt(\psi_2,H)=1$
  \includegraphics[width=0.3\linewidth]{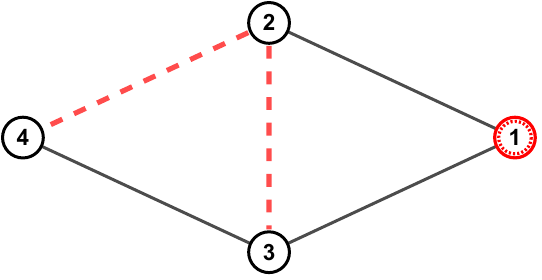} \\[5pt]
  $\nodcnt(\psi_3,H)=3$
  \includegraphics[width=0.3\linewidth]{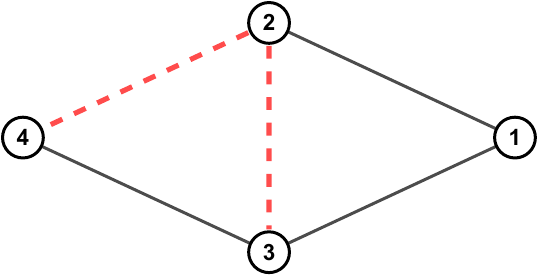}
  \quad
  $\nodcnt(\psi_4,H)=4$
  \includegraphics[width=0.3\linewidth]{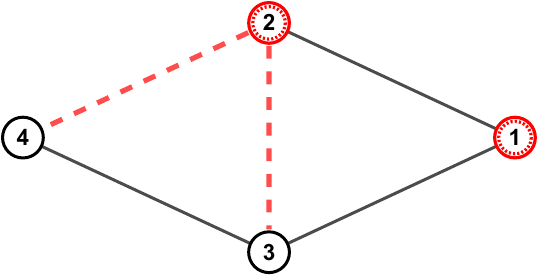}
  \caption{Eigenvector data used for computing the nodal count
    $\nodcnt(\psi_k,H)$, as defined in \eqref{eq:nodal_count_def}.  A
    graph edge $(u,v)$ is colored black (solid) if $H_{u,v}>0$ and red
    (dashed) if $H_{u,v}<0$.  A vertex $u$ is colored black (solid) if
    $\psi_u>0$ and red (composite solid and dashed) if $\psi_u<0$.}
  \label{fig:MHExample}
\end{figure}

The matrix $H$ has four eigenvectors whose signs are summarized in
\Cref{fig:MHExample}.  It is apparent that
\begin{equation}
  \label{eq:nocnt_ex}
  \nodcnt(\psi_1,H)=1,\quad
  \nodcnt(\psi_2,H)=1,\quad
  \nodcnt(\psi_3,H)=3,\quad
  \nodcnt(\psi_4,H)=4. 
\end{equation}
\Cref{fig:MagHam} depicts the eigenvalues of the magnetic Hamiltonian
$H_\alpha$ as a function of the flux coordinates $k_1,k_2$.  We
can see that $(0,0)$ is a saddle point (index 1) except for
$\Lambda_2$, where it is a minimum (index 0), matching
\Cref{cor:nodal_magnetic}, equation \eqref{eq:Morse_index}:
\begin{align}
  \label{eq:MorseLambda_ex1}
  &n_-(\Hess\Lambda_1)=1 = 1 - (1-1),
  &&n_-(\Hess\Lambda_2)=0 = 1 - (2-1),\\
  \label{eq:MorseLambda_ex2}
  &n_-(\Hess\Lambda_3)=1 = 3 - (3-1),
  &&n_-(\Hess\Lambda_4)=1 = 4 - (4-1).
\end{align}

\begin{figure}
  \centering
  \includegraphics[width=0.45\linewidth]{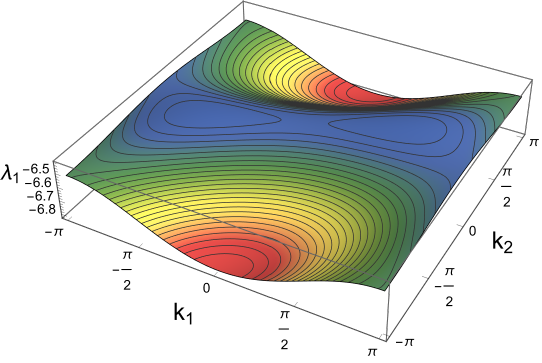}
  \includegraphics[width=0.45\linewidth]{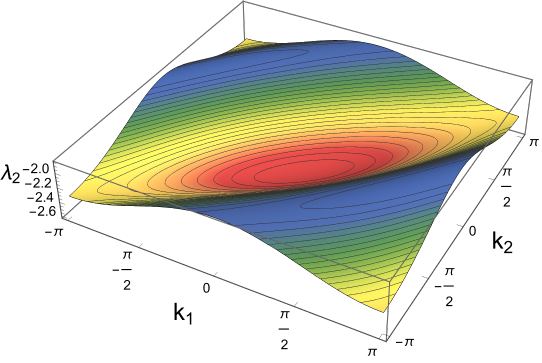}\\
  \includegraphics[width=0.45\linewidth]{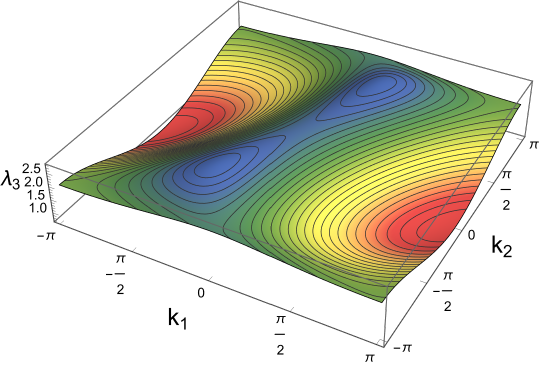}
  \includegraphics[width=0.45\linewidth]{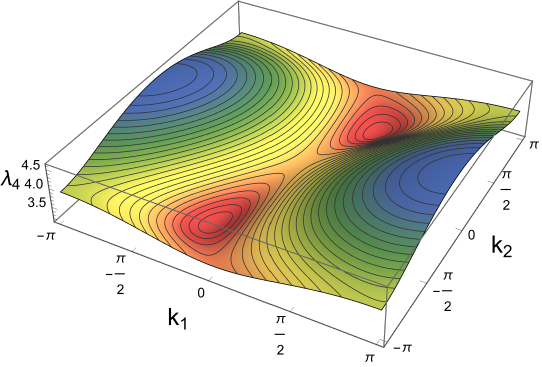}
  \caption{The eigenvalues of the magnetic perturbation Hamiltonian
    as a function of the fluxes $k_1$, $k_2$.}
  \label{fig:MagHam}
\end{figure}

\begin{figure}
  \centering
  \includegraphics[width=0.45\linewidth]{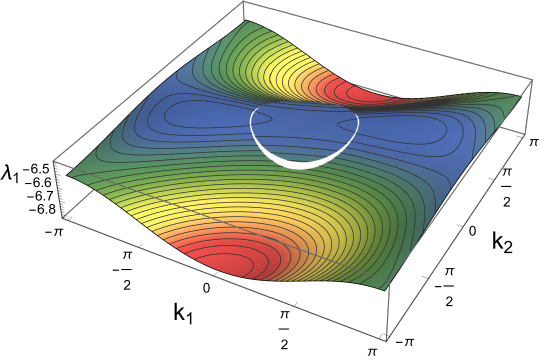}
  \includegraphics[width=0.45\linewidth]{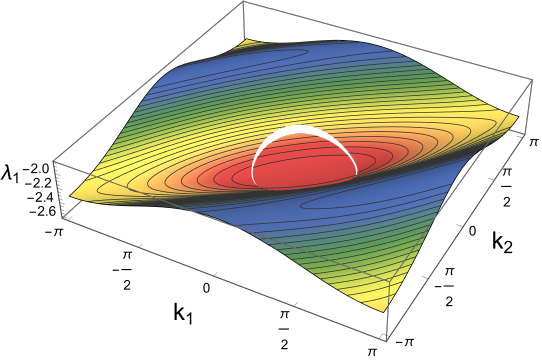}  
  \caption{The eigenvalues $\Lambda_1$ and $\Lambda_2$ of $H_\alpha$,
    equation~\eqref{eq:H_example}, for $k_1^2+k_2^2>1$, and the
    quadratic approximation at $(k_1,k_2)=(0,0)$, shown for $k_1^2+k_2^2<1$, computed using ~\eqref{eq:Hess_flux}.}
  \label{fig:MagHess}
\end{figure}

Finally, the Hessian of $\Lambda_k$ in terms of the flux coordinates
(cf.\ equation~\eqref{eq:Hess_flux_coords}) is
\begin{equation}
  \label{eq:Hess_flux}
  \frac12\Hess \Lambda(\vec{k}_1,\vec{k}_2)
  = \left<\vec{k}_1, \left(C^\top R C\right)^{-1} \vec{k}_2 \right>_{\R^2}.
\end{equation}
For the first two eigenvectors of $H$, the former is given by:
\begin{align}
  \label{eq:intersection_form_ex1}
  &\frac12\Hess \Lambda_1(\vec{k}_1,\vec{k}_2) = \begin{pmatrix}
      -0.0066& 0.0469\\
      0.0469 & -0.1106 \\
    \end{pmatrix},
  &&\frac12\Hess \Lambda_2(\vec{k}_1,\vec{k}_2) = \begin{pmatrix}
      0.1455 & -0.1326 \\
      -0.1326& 0.1861 \\
    \end{pmatrix},
\end{align}
where $R_1$, $R_2$ have been computed from the eigenvectors $\psi_1$,
$\psi_2$ by the usual formula, equation~\eqref{eq:the_r_formula}.  We have computed $\frac12\Hess \Lambda(\vec{k}_1,\vec{k}_2)$ both directly, by numerically differentiation of the eigenvalues, and via \Cref{eq:Hess_flux}. We
remark that $n_{-}(C^\top R_1 C)=1$ and $n_{-}(C^\top R_2 C)=0$, in
agreement with the results in
\eqref{eq:MorseLambda_ex1}-\eqref{eq:MorseLambda_ex2}.

Using the Hessian to compute the quadratic approximation to $\Lambda_1$ and $\Lambda_2$
in the coordinates $\vec{k} = (k_1,k_2)$ results in the graphs in
\Cref{fig:MagHess} (these correspond to the top two graphs in
\Cref{fig:MagHam}).  The quadratic approximation performs remarkably
well, which can be partially explained by the absence of cubic terms
in the Taylor series, due to the symmetry \eqref{eq:real_symm}.

\section{Oscillation and the cycle intersection form}
\label{sec:aux}

In this section we collect some preliminaries and then prove
\Cref{thm:main} using the language of linear algebra; the language of
forms will be used to establish \Cref{thm:magHessian}.

\subsection{Compression, inertia and the Schur complement}
\label{sec:sylvester}

For an operator $T:\R^m \to \R^m$, its \emph{compression} $[T]_V$ to a
subspace $V$ is the composition
\begin{equation}
  \label{eq:compression_diagram}
  V \overset{\iota}{\longrightarrow} \R^m
  \overset{T}{\longrightarrow} \R^m \overset{P}{\longrightarrow}
  V,
\end{equation}
where $\iota$ is the inclusion and $P=\iota^*$ is the orthogonal
projection. If $T$ is symmetric, and $\mathsf{T}$ denotes the 
corresponding quadratic form, then the restriction of $\mathsf{T}$ to $V$ is 
the quadratic form of the compression $[T]_V$.

We say that $U$ is an \emph{orthonormal frame} for the subspace
$V\subset \R^m$ if $U$ is an $m\times n$ matrix whose columns
$\{u_1,\ldots, u_n\}$ give an orthonormal basis of $V$.


With respect to the basis given by the columns of $U$, the matrix
representation of the compression of $T:\R^m \to \R^m$ to $V$ is
\begin{equation}
  \label{eq:matrix_compression}
  [T]_V = U^* T U.
\end{equation}
Naturally, the eigenvalues of $U^* T U$ do not depend on the choice of
the particular orthonormal basis.  Replacing the orthonormal frame $U$
with an arbitrary frame $S$ of the same subspace may change the
eigenvalues of $S^* T S$, but preserves their sign, by Sylvester's law
of inertia \cite[Thm.~4.5.8]{HornJohnson}.


We now state Sylvester's law of inertia in a somewhat more general
form than what is found in textbooks.

\begin{proposition}[Sylvester's law of inertia with a general
  transformation]
  \label{lem:sylvester}
  Let $H$ be $n\times n$ hermitian and let $S$ be
  $n \times m$.  Then
  \begin{align}
    \label{eq:sylvester_ind}
    & n_\pm(S^* H S) = n_\pm\Big( [H]_{\Ran S} \Big),\\
    \label{eq:sylvester_ker}
    & n_0(S^* H S) = n_0\Big( [H]_{\Ran S} \Big) + \dim \ker S.
  \end{align}
  In particular, if $V$ is an arbitrary subspace and $T$ is a linear transformation,
  \begin{equation}
    \label{eq:sylvester_restrictions}
    n_\pm\Big( \big[H\big]_{TV} \Big)
    = n_\pm\Big( \big[T^* H T\big]_V \Big).
  \end{equation}
\end{proposition}

\begin{proof}
  If $S$ is invertible, $[H]_{\Ran S} = H$ and $\dim\ker
  S = 0$, resulting in the classical law of inertia
  \cite[Thm.~4.5.8]{HornJohnson}.  Taking this as proven, we discuss
  the extension to matrices $S$ that are not necessarily invertible or even
  square.

  Let $s = \dim \Ran S$ and let $n\times s$ matrix $U$ be an
  orthonormal frame for $\Ran S$.  Choose $\{w_1, \ldots, w_s\}$ such
  that $Sw_j = u_j$, the $j$-th column of $U$.  Choose the vectors
  $\{w_{s+1},\ldots, w_m\}$ to be a basis of the space $\ker S$ (the
  dimensions are right by the rank-nullity theorem).  Form the matrix
  $W$ using the vectors $w_j$ as columns which can be easily seen to
  be linearly independent (assume the contrary and apply $S$).  Since
  $W$ is invertible, $S^* H S$ and $W^* S^* H S W$ have identical
  indices $n_-$, $n_+$ and $n_0$, by the ``invertible'' Sylvester's
  theorem.  An explicit computation yields
  \begin{equation}
    W^* S^* H S W
    =
    \begin{pmatrix}
      U^* \\ 0
    \end{pmatrix}
    H
    \begin{pmatrix}
      U & 0
    \end{pmatrix}
    =
    \begin{pmatrix}
      U^* H U & 0 \\
      0 & 0
    \end{pmatrix},
  \end{equation}
  leading to \eqref{eq:sylvester_ind}-\eqref{eq:sylvester_ker}.
  To verify \eqref{eq:sylvester_restrictions}, observe that if $V =
  \Ran S$ then $TV = \Ran TS$ and
  \begin{equation}
    \label{eq:restriction_change}
    n_\pm\Big( \big[T^* H T\big]_V \Big)
    = n_\pm\left(S^*T^* H T S\right)
    = n_\pm\Big( \big[H\big]_{TV} \Big).
  \end{equation}
\end{proof}


\begin{corollary}[Haynsworth inertia additivity formula \cite{Hay_laa68}]
  \label{cor:Haynsworth}
  For a hermitian block matrix
  \begin{equation}
    \label{eq:H_block}
    H =
    \begin{pmatrix}
      A & C \\ C^* & D
    \end{pmatrix}
  \end{equation}
  with invertible $A$,
  \begin{align}
    \label{eq:Haynsworth}
    n_-(H) &= n_-(A) + n_-(H/A),\\
    \label{eq:Haynsworth0}
    n_0(H) &= n_0(A) + n_0\left(H/A\right),
  \end{align}
  where $H/A := D - C^* A^{-1} C$ is the \emph{Schur
    complement} of the block $A$ in $H$.
\end{corollary}

\begin{proof}
  Use \Cref{lem:sylvester} with the invertible
  \begin{equation}
    \label{eq:SforHaynsworth}
    S :=
    \begin{pmatrix}
      I & -A^{-1} C \\
      0 & I
    \end{pmatrix},
  \end{equation}
  which makes $S^*HS$ block diagonal with the blocks $A$ and $H/A$.
\end{proof}

Extensions of \eqref{eq:Haynsworth}-\eqref{eq:Haynsworth0} to
non-invertible $A$ have been proposed in many papers with probably the
most general formula due to Maddocks \cite[Thm.~6.1]{Mad_laa88} (see
also \cite[App.~A]{BerCanCoxMar_paa22}).  We will need its special
case with $A=0$, which is important due to its appearance as the
\emph{bordered Hessian} in Lagrange optimization.  The formula we need
is due to Jongen, M\"obert, R\"uckmann and Tammer
\cite[Thm.~2.1]{Jon+_laa87}, which extends an earlier formula of Han
and Fujiwara \cite[Thm.~3.4]{HanFuj_laa85}.  For completeness we
derive it from \Cref{cor:Haynsworth} using nothing more than the
continuity of eigenvalues under perturbation.

\begin{corollary}[Inertia for bordered Hessians
  \cite{HanFuj_laa85,Jon+_laa87}]
  \label{cor:JMRT}
  Let
  \begin{equation}
    \label{eq:bordered_Hessian}
    M =
    \begin{pmatrix}
      A & C \\
      C^* & 0_\beta
    \end{pmatrix}
  \end{equation}
  be hermitian, with $0_\beta$ denoting zero matrix of size
  $\beta\times\beta$. Then
  \begin{align}
    \label{eq:n-bordered}
    n_-(M)
    &=
    n_-\Big( [A]_{\ker C^*} \Big)
      + \dim \Ran C, \\
    \label{eq:n0bordered}
    n_0(M)
    &=
    n_0\Big( [A]_{\ker C^*} \Big)
    + \beta - \dim \Ran C.
  \end{align}
\end{corollary}

\begin{proof}
  For \(\varepsilon>0\), set
  \[
    M_\varepsilon :=
    \begin{pmatrix}
      A & C \\
      C^* & \varepsilon I
    \end{pmatrix}.
  \]
  Since \(M_\varepsilon\) is a non-negative perturbation of $M$,
  continuity of eigenvalues gives \(n_-(M_\varepsilon)=n_-(M)\) for
  all sufficiently small \(\varepsilon>0\). By \Cref{cor:Haynsworth}
  (with a suitable rearrangement),
  \begin{equation}
    \label{eq:JMRT_eps}
    n_-(M)
    =
    n_-(M_\varepsilon)
    =
    n_-(\varepsilon I) +
    n_-\left(A-\frac1\varepsilon CC^*\right)
    =
    n_-\left(\varepsilon A-CC^*\right).
  \end{equation}
  Writing the last matrix in block form with respect to the decomposition
  \(\ker C^*\oplus \Ran C\),
  \[
    \varepsilon A-CC^*
    =
    \begin{pmatrix}
      \varepsilon [A]_{\ker C^*} & \varepsilon A_1 \\
      \varepsilon A_1^* & -C_2+\varepsilon A_2
    \end{pmatrix},
  \]
  where \(C_2 = \left[CC^*\right]_{\Ran C}>0\). Applying
  \Cref{cor:Haynsworth} once more gives
  \begin{equation}
    \label{eq:Haynsagain}
    n_-\left(\varepsilon A-CC^*\right)
    =
    n_-(-C_2+\varepsilon A_2)
    +
    n_-\left(
    [A]_{\ker C^*}
    +
    \varepsilon A_1 (C_2-\varepsilon A_2)^{-1} A_1^*
    \right).
  \end{equation}
  For small \(\varepsilon>0\), the block \(-C_2+\varepsilon A_2\) is
  negative definite, while the second term in \eqref{eq:Haynsagain} is
  a small non-negative perturbation of \([A]_{\ker C^*}\).  We arrive at
  \begin{equation}
    \label{eq:n-bordered-proof}
    n_-(M) =
    \dim \Ran C
    +
    n_-\left([A]_{\ker C^*}\right),
  \end{equation}
  which is \eqref{eq:n-bordered}.  To prove \eqref{eq:n0bordered},
  apply \eqref{eq:n-bordered-proof} to \(-M\), getting
  \[
    n_+(M)=n_+\Big( [A]_{\ker C^*} \Big)+\dim \Ran C .
  \]
  Writing the size of the matrix \(M\) as
  \(\dim\ker C^*+\dim\Ran C+\beta\), we get
  \begin{align*}
    n_0(M)
    &=
    \left(\dim\ker C^*+\dim\Ran C+\beta\right)
    - n_-(M) - n_+(M) \\
    &=
    \bigg(\dim\ker C^*
    -
    n_-\big([A]_{\ker C^*}\big)
    -
    n_+\big([A]_{\ker C^*}\big)\bigg)
    + \beta-\dim\Ran C \\
    &=
    n_0\big([A]_{\ker C^*}\big)+\beta-\dim\Ran C .
  \end{align*}
\end{proof}

\subsection{Coboundary operator and the ground-state transform}
\label{sec:coboundary}

Define the \emph{coboundary operator} $d: \R^V \to \R^\E$ by
\begin{equation}
  \label{eq:cbd_op_def}
  (d\theta)_{(u,v)} = \theta_v - \theta_u.
\end{equation}
In graph theory $d$ is usually called the \emph{oriented incidence
  matrix}, but in \Cref{sec:chains} we will identify the spaces
$\R^V$ and $\R^\E$ with the spaces of cochains, $C^0(G, \R)$ and
$C^1(G, \R)$, which justifies the name ``coboundary'' for $d$.

The adjoint of $d$ with respect to the standard inner products on
$\R^V$ and $\R^\E$ is
\begin{equation}
  \label{eq:bd_op_def}
  d^*: \R^\E \to \R^V,
  \qquad
  (d^* f)_u = \sum_{v: (v,u)\in \E} f_{(v,u)} - \sum_{w: (u,w)
    \in \E} f_{(u,w)}.
\end{equation}
Comparing \eqref{eq:bd_op_def} with \eqref{eq:flow_condition} and
using the fundamental subspaces theorem, we have
\begin{equation}
  \label{eq:fund_subspaces}
  Z = \Ran C = \ker d^*,
  \qquad
  Z^\perp = \ker C^\top = \Ran d = d \R^V,
\end{equation}
where $C$ is a frame for the cycle space $Z$.  The space $Z^\perp = d
\R^V$ is commonly called the \emph{cut space} of the graph.

The following lemma, which can be traced through
\cite[Sec.~4]{Col_apde13} to \cite[Proof of Thm.~(2,2)]{Fie_cmj75},
may be viewed as a generalization of the fact that $d^* d$ is the
standard graph Laplacian.  

\begin{lemma}[Ground-state transform]
  \label{lem:ground_state_transform}
  Let \(G=(V,E)\) be a simple graph with a fixed orientation
  \(\E=\vec E\), and let real symmetric $H$ be supported on \(G\). Let
  \(\psi\) be a real eigenvector of \(H\) with eigenvalue \(\lambda\), and
  let \(\Psi\) be the operator of multiplication by \(\psi\) on
  \(\R^V\), i.e.\ the \(|V|\times |V|\) diagonal matrix with entries
  \(\psi_u\). Define the diagonal matrix \(\Phi\) on \(\R^\E\) by
  \[
    \Phi_{e,e} = -\psi_u H_{u,v}\psi_v,
    \qquad e=(u,v)\in \E .
  \]
  Then
  \begin{equation}
    \label{eq:ground_state_transform}
    d^* \Phi d = \Psi^* (H-\lambda) \Psi.
  \end{equation}
\end{lemma}

\begin{remark}
  The name ``ground-state transform'' is borrowed from the analysis of
  Schr\"odinger operators and quantum field theory
  \cite[Sec.~4]{Gro_ajm75}, where it denotes the conjugation of $H$ by
  its positive ground-state eigenfunction .  In
  \Cref{lem:ground_state_transform}, $\psi$ is an arbitrary
  eigenvector, not necessarily a ground state, and its entries need
  not be non-zero.  However, it is essential that $H$ is real.
\end{remark}

\begin{proof}
  An explicit computation shows that for $u \neq v$,
  $(d^* \Phi d)_{u,v} = \psi_u H_{u,v} \psi_v$.  For the diagonal
  entries we use $H\psi=\lambda\psi$ to compute
  \begin{equation*}
    (d^* \Phi d)_{u,u} = -\sum_{v\sim u} \psi_u H_{u,v} \psi_v
    = \psi_u \left( H_{u,u}\psi_u - \sum_{v \in V} H_{u,v} \psi_v 
    \right)
    = \psi_u \left(H_{u,u} \psi_u  - \lambda \psi_u\right).
  \end{equation*}
  These are precisely the diagonal entries of
  \(\Psi^*(H-\lambda)\Psi\), completing the proof.
\end{proof}

\begin{corollary}[Inertia on the cut space]
  \label{cor:ind_Phi}
  With the notation as in \Cref{lem:ground_state_transform} and assuming
  additionally that the eigenvector $\psi$ has no zero entries,
  \begin{align}
    \label{eq:Phi_ind}
    &n_-\left( [\Phi]_{d \R^V} \right) = n_-(H-\lambda), \\
    \label{eq:Phi_ker}
    &n_0\left( [\Phi]_{d \R^V} \right) = n_0(H-\lambda) - c,
  \end{align}
  where $c=c(G)$ is the number of connected components of the graph $G$.
\end{corollary}

\begin{proof}    
  We combine \Cref{lem:ground_state_transform} and \Cref{lem:sylvester} to
  obtain
  \begin{equation}
    \label{eq:Phi_ind_comp}
    n_-\left( [\Phi]_{\Ran d} \right) = n_-\left( d^* \Phi d \right)
    = n_-\left(\Psi^* (H-\lambda) \Psi \right) = n_-(H-\lambda),
  \end{equation}
  since $\Psi$ is invertible.  Similarly, 
  \begin{equation}
    \label{eq:Phi_ker_comp}
    n_0\left( [\Phi]_{\Ran d} \right)
    = n_0\left(d^* \Phi d\right) - \dim \ker d
    = n_0\left(\Psi^* (H-\lambda) \Psi \right) - c = n_0(H-\lambda) - c,
  \end{equation}
  where we used that $\dim \ker d$ is the number of connected
  components of the graph $G$.
\end{proof}

\subsection{Proof of \texorpdfstring{\Cref{thm:main}}{the main theorem}}
\label{sec:proof_schur}

The idea of the proof is to compare the results of two different
applications of the Haynsworth identity to the matrix
\begin{equation}
  \label{eq:block_matrix}
  M :=
  \begin{pmatrix}
    \Phi & C \\
    C^\top & 0_\beta
  \end{pmatrix},
\end{equation}
where $0_\beta$ is the $\beta \times \beta$ zero matrix, and $C$ is a
frame for the cycle space $Z$.

\begin{proof}[Proof of \Cref{thm:main}]
  Since \(\psi\) has no zero entries and \(H\) is strictly supported
  on \(G\), the diagonal matrix \(\Phi\) is invertible.  We now apply
  the Haynsworth identity, \Cref{cor:Haynsworth}, to the matrix $M$
  with the Schur complement $M/\Phi = 0_\beta - C^\top \Phi^{-1} C$,
  \begin{align}
    \label{eq:schur1}
    &n_-(M) = n_-(\Phi) + n_-\left(- C^\top \Phi^{-1} C \right), \\
    \label{eq:schur1_ker}
    &n_0(M) = n_0(\Phi) + n_0\left(- C^\top \Phi^{-1} C \right)
      =  n_0\left(C^\top \Phi^{-1} C \right).
  \end{align}
  
  Next, we use \Cref{cor:JMRT} with $\dim \Ran C = \beta$,
  \begin{align}
    \label{eq:Han_Fujiwara}
    n_-(M)
    &= n_-\left( [\Phi]_{\ker C^\top} \right) + \dim\Ran C \\
    \nonumber
    &=  n_-\left( [\Phi]_{d\R^V} \right) + \beta
      = n_-(H-\lambda) +\beta, \\
    \label{eq:Han_Fujiwara_ker}
    n_0(M)
    &= n_0\left( [\Phi]_{\ker C^\top} \right) + \beta - \dim\Ran C \\
    \nonumber
    &= n_0\left( [\Phi]_{d\R^V} \right)
      = n_0(H-\lambda) - c.
  \end{align}
  To evaluate \eqref{eq:Han_Fujiwara} and \eqref{eq:Han_Fujiwara_ker} we
  used \eqref{eq:Phi_ind} and \eqref{eq:Phi_ker}.
  
  Comparing \eqref{eq:Han_Fujiwara_ker} and \eqref{eq:schur1_ker}, we
  conclude
  \begin{equation}
    \label{eq:n0computed}
    n_0\left(C^\top \Phi^{-1} C \right) = n_0(H-\lambda) - c.
  \end{equation}
  Comparing \eqref{eq:Han_Fujiwara} and
  \eqref{eq:schur1}, we get
  \begin{align}
    \label{eq:and_done}
    n_-(H-\lambda) +\beta
    &= n_-(\Phi) + n_-\left(- C^\top \Phi^{-1} C \right) \\
    &= n_-(\Phi) + \beta - n_-\left(C^\top \Phi^{-1} C \right)
      - n_0\left(C^\top \Phi^{-1} C \right).
  \end{align}
  Rearranging, we get
  \begin{equation}
    \label{eq:main_formula_other}
    n_-(\Phi)
    = n_-(H-\lambda)
    + n_0\left(C^\top \Phi^{-1} C \right)
    + n_-\left(C^\top \Phi^{-1} C \right).
  \end{equation}
  By definition, $\nodcnt(\psi,H) = n_-(\Phi)$.  We now use
  \eqref{eq:n0computed} and
  $k^\uparrow := n_-(H-\lambda) + n_0(H-\lambda)$, to arrive to
  \begin{equation}
    \label{eq:the_end}
    \nodcnt(\psi,H)
    = k^\uparrow - c + n_-\left(C^\top \Phi^{-1} C \right).
  \end{equation}
  
  Since $\Phi^{-1}=R$, where $R$ is the diagonal matrix of the weights $r_e
  = - \left(\psi_u H_{u,v} \psi_v\right)^{-1}$, the matrix $C^\top \Phi^{-1} C =
  C^\top R C$ is the matrix representation of the form $\mathcal{X}_r$
  in the cycle basis specified by $C$.  Equations \eqref{eq:n0computed}
  and \eqref{eq:the_end} are therefore equivalent to
  \eqref{eq:main_formula_nullity} and \eqref{eq:main_formula_surplus}.
\end{proof}

\section{Magnetic Hessian and the nodal-magnetic theorem}
\label{sec:magnetic}

In this section we explore the relationship between the \emph{weighted
  cycle intersection form} and the Hessian of the eigenvalue
$\lambda_k$ viewed as a function on the \emph{torus of magnetic
  perturbations}, linking our results with the previous studies
\cite{AloBerGor_prep25, AloGor_jst23,AloGor_prep24,Ber_apde13,Col_apde13}.

\subsection{Chains and cochains on a graph}
\label{sec:chains}

The space $C_0(G,\R)$ of
$0$-chains consists of formal linear combinations of vertices.
The space $C_1(G,\R)$ of 1-chains consists of formal linear
combinations of edges $e = (u,v)\in \vec{E}$. The
boundary map $\partial : C_1(G,\R) \to C_0(G,\R)$ is defined via
\begin{equation}
  \label{eq:boundary_map}
  \partial e = (v) - (u), \qquad \text{where }e = (u,v)\in\E,
\end{equation}
extended by linearity.
The kernel of $\partial$ consists of {\em cycles} and, since there are no
2-chains, coincides with the first homology group
\begin{equation}
  \label{eq:homH1}
  H_1(G, \R) = \ker \partial / \Ran 0
  = \ker \partial = Z.  
\end{equation}

The spaces dual to $C_0(G,\R)$ and $C_1(G,\R)$ are the spaces of
cochains $C^0(G,\R)$ and $C^1(G,\R)$ (or 0-forms and 1-forms).  The
space $C^0(G,\R)$ can be viewed as the space of maps from $V$ to $\R$,
or $\R^V$.  The space $C^1(G,\R)$ is the space of maps from
\emph{oriented} edges $\vec{E}$ to $\R$.  The duality pairing is
\begin{equation}
  \label{eq:pairingC1}
  \alpha \cdot y = \sum_{e\in\vec{E}} \alpha(e) y_e,
  \qquad
  \text{where } C^1(G,\R) \ni \alpha : e\mapsto \alpha(e),\ \ 
  y = \sum_{e\in\vec{E}} y_e\, e \in C_1(G,\R).
\end{equation}
The pairing between $C^0$ and $C_0$ is defined analogously.


With these pairings, the operator dual to $\partial$ is the coboundary
operator\footnote{Under the natural identification of both $C_0$ and
  $C^0$ with $\R^V$, and both $C_1$ and $C^1$ with $\R^E$, the duality
  pairing becomes the Euclidean inner porduct.  The action of $d$ is
  then given by \eqref{eq:cbd_op_def} and the action of $\partial$ is
  its adjoint $d^*$ in \eqref{eq:bd_op_def}.}
$d: C^0(G,\R) \to C^1(G,\R)$, easily verified to be
\begin{equation}
  \label{eq:coboundary_def}
  (d \theta)(e) = \theta(v) - \theta(u),
  \quad\text{where }e = (u,v)\in\E,
\end{equation}
same as defined in \eqref{eq:cbd_op_def}.  Because $C^2(G, \R) = 0$,
the first cohomology is
\begin{equation}
  \label{eq:cohomH1}
  H^1(G, \R) = \ker 0 / \Ran d
  = C^1(G,\R) / d\R^V.
\end{equation}
Since $\ker \partial$ is the annihilator of $\Ran d$, two cochains
$\alpha_1,\alpha_2 \in C^1(G, \R)$ belong to the same cohomology class
if and only if they take the same value on every cycle
$\gamma \in H_1(G,\R) \subset C_1(G,\R)$.  In other words,
\begin{equation}
  \label{eq:H_duality}
  [\alpha_1] = [\alpha_2]
  \quad \Longleftrightarrow \quad
  \alpha_1\cdot\gamma = \alpha_2 \cdot \gamma
  \quad \text{for any }\gamma \in H_1(G,\R),  
\end{equation}
where $[\alpha]\in H^1(G,\R)$ denotes the equivalence class of
$\alpha\in C^1(G,\R)$ and $\cdot$ is the duality pairing
\eqref{eq:pairingC1}.

Once a frame \(C\) or, equivalently, a basis
$\{\gamma_1, \ldots, \gamma_\beta\}$, has been chosen for the cycle
space \(Z=H_1(G,\R)\), it determines coordinates on \(H^1(G,\R)\). The
\emph{flux coordinates} of a cohomology class
\([\alpha]\in H^1(G,\R)\) are
\begin{equation}
  \label{eq:flux_coord_def}
  \alpha \cdot C
  := \big(\alpha\cdot \gamma_1\ \cdots\ \alpha\cdot \gamma_\beta\big).
\end{equation}
By \eqref{eq:H_duality}, these coordinates are well-defined and
separate cohomology classes. In other words, the pairing
between \(C^1\) and \(C_1\) descends to a non-degenerate pairing
between \(H^1(G,\R)\) and \(H_1(G,\R)\), and \eqref{eq:flux_coord_def}
are the coordinates in the dual basis.

Finally, while there is no \emph{canonical} identification of
$H^1(G, \R)$ with a subspace of $C^1(G,\R)$, in \Cref{lem:Hodge_flat}
below we will identify a subspace particularly suitable for the
calculation of the Hessian of the eigenvalue $\lambda_k$.

\subsection{Tori of phase and magnetic perturbations}
\label{sec:phase_pert}

Recall that an $n \times n$ matrix $H$ is {\em supported on $G$}
(resp.~{\em strictly supported on $G$}) if, for any $u \ne v$,
$H_{u,v} \ne 0 \implies u\sim v$
(resp.~$H_{u,v} \ne 0 \iff  u\sim v$).  Associated to a graph $G$
there are vector spaces of matrices:
\begin{itemize}
\item $\cH (G)$ is the set of \emph{complex hermitian} matrices supported on $G$;
\item $\cS(G)\subset \cH (G)$ is the set of \emph{real symmetric} matrices supported on $G$;
\item $\cA(G)$ is the set of \emph{real antisymmetric} matrices supported on $G$.
\end{itemize}

 Assume the real symmetric matrix $H\in\cS(G)$ is \emph{strictly}
 supported on $G$.  The \emph{torus of phase perturbations} of $H$ is
 the family of matrices in $\cH(G)$ given by
\begin{equation}
  \label{eq:Mh_def}
  \cP_H := \{H_A\in\cH(G) \colon
  (H_A)_{u,v} := e^{ia_{u,v}} H_{u,v}, \ 
  A = \left(a_{u,v}\right) \in \cA(G)\}.
\end{equation}
Topologically $\cP_H$ is a torus, since each $a_{rs}$ is defined only
modulo $2\pi$, with the covering space $\cA(G)$.  It is a
differentiable manifold with the tangent space naturally
identified with $\cA(G)$.

Real antisymmetric matrices $\cA(G)$ on the graph $G$ are naturally
identified with the cochains $C^1(G,\R)$, via
$C^1(G,\R) \ni \alpha \mapsto A =(a_{u,v}) \in \cA(G)$ with the matrix
elements
\begin{equation}
  \label{eq:1formA}
  a_{u,v} =
  \begin{cases}
    \alpha(e), & e = (u,v) \in \vec{E},\\
    -\alpha(e), & e = (v,u) \in \vec{E}, \\
    0, & u\not\sim v.
  \end{cases}
\end{equation}
With this identification in mind, we will use the notation $H_\alpha$
to mean the corresponding phase perturbation $H_A$ defined in
\eqref{eq:Mh_def}.

\begin{lemma}
  \label{lem:gauge}
  The matrices $H_{\alpha}$ and $H_{\alpha+d\theta}$ are unitarily
  equivalent for any $\theta \in \R^V$.  Consequently, the eigenvalue
  function $\Lambda(\alpha) = \lambda_k(H_\alpha)$ depends only on the
  cohomology class $[\alpha]\in H^1(G,\R)$ of $\alpha \in C^1(G, \R)$.
\end{lemma}

\begin{proof}
  Direct computation shows that
  \begin{equation}
    \label{eq:gauge_equiv_H}
    H_{\alpha+d\theta} = e^{-i\theta} H_\alpha e^{i\theta},
  \end{equation}
  where $e^{i\theta}$ is the diagonal matrix
  $\diag\left(e^{i\theta_1}, e^{i\theta_2},\cdots, e^{i\theta_n}\right)$.
\end{proof}

Taking the quotient of $\cP_H$ by this equivalence we obtain the
\emph{torus of magnetic perturbations} of $H$, which is discussed in
detail in, e.g., \cite[Sec.~5]{AloBerGor_prep25}.

\subsection{Musical isomorphisms and a Hodge-type splitting}
\label{sec:musical_Hodge}

We now recall the \emph{musical isomorphisms}, starting with ``flat''.
Let $\mathcal{B}$ be a non-degenerate bilinear form on a vector space
$V$, and let $V^*$ denote the dual space.  The \emph{flat} isomorphism
$v\mapsto v^\flat$ between $V$ and $V^*$ is defined by
\begin{equation}
  \label{eq:flat_def}
  v^\flat \cdot u = \mathcal{B}(v, u),
\end{equation}
where $\cdot$ denotes the duality pairing between $V^*$ and $V$.

The inverse of the flat isomorphism is the \emph{sharp} isomorphism $V^*
\to V$, defined by
\begin{equation}
  \label{eq:sharp_def}
  f \mapsto f^\sharp,
  \quad \text{such that} \quad
  f\cdot u = \mathcal{B}(f^\sharp, u)
  \quad \text{for all } u\in V.
\end{equation}
The \emph{dual form} $\mathcal{B}^\sharp$ on $V^*$
is then defined as
\begin{align}
  \label{eq:dual_form_def}
  \mathcal{B}^\sharp(f_1, f_2)
  &= \mathcal{B}(f_1^\sharp, f_2^\sharp)
  \\
  \label{eq:dual_form_def_alt}
  &= f_1 \cdot f_2^\sharp,
\end{align}
where the second form is a consequence of \eqref{eq:sharp_def}.
If $\mathcal{B}$ has a matrix $B$ in a certain basis, the matrix of
$\mathcal{B}^\sharp$ in the dual basis is simply $B^{-1}$.

If the space $V$ is equipped with an inner product
$\left<\cdot, \cdot\right>$ (i.e., a non-degenerate bilinear form
which is in addition positive definite), then the corresponding
musical isomorphisms are just the Riesz representations; for example,
$v^\flat$ is the functional $u \mapsto \left<v,u\right>$.

To give a preview why the cycle intersection form (or, rather, its
dual) arises in computation of the Hessian of the eigenvalue function
$\Lambda(\alpha)$, we have the following lemma.
\begin{lemma}
  \label{lem:HessQF}
  For any $\psi$ with non-zero entries and real symmetric $H$ strictly
  supported on $G$,
  \begin{equation}
    \label{eq:HessQF}
    \frac12\frac{\partial^2}{\partial t_1 \partial t_2}
    \left<\psi, H_{t_1\alpha_1+t_2\alpha_2}\psi\right>\Big|_{t_1,t_2=0}
    = \mathcal{R}^\sharp(\alpha_1, \alpha_2),
  \end{equation}
  where $\mathcal{R}^\sharp$ is the dual form to $\mathcal{R}$ defined
  in \eqref{eq:Rform_def} with the coefficients
  $r_e = -\left(\psi_u H_{u,v} \psi_v\right)^{-1}$, $e=(u,v)\in\E$, as in
  \Cref{thm:magHessian}.
\end{lemma}

\begin{remark}
  \label{rem:HessQF}
  Despite similarities with the main claim of~\Cref{thm:magHessian},
  equation \eqref{eq:Hess}, the claim in \eqref{eq:HessQF} is
  different (and much easier to prove): the vector $\psi$ is held
  fixed and the right-hand side depends on the entirety of
  $\alpha_{1,2}$ rather than their equivalence classes.
\end{remark}

\begin{proof}
  This is an explicit computation,
  \begin{align*}
    &\frac12\frac{\partial^2}{\partial t_1 \partial t_2}
      \left<\psi, H_{t_1\alpha_1+t_2\alpha_2}\psi\right>\Big|_{t_1,t_2=0}\\
    &\quad
      = \frac12\frac{\partial^2}{\partial t_1 \partial t_2} \left[
      \sum_{e=(u,v)\in \E} \Big(
      \psi_u e^{it_1 \alpha_1(e)+it_2 \alpha_2(e)}H_{u,v} \psi_v
      +
      \psi_v e^{-it_1\alpha_1(e)-it_2\alpha_2(e)}H_{v,u}\psi_u
      \Big)\right]_{t_1,t_2=0}\\
    &\quad
      = \sum_{e=(u,v)\in \E} \alpha_1(e) \left(-\psi_u H_{u,v}\psi_v\right) \alpha_2(e)
      = \sum_{e\in \E} \alpha_1(e) \frac{1}{r_e} \alpha_2(e)
      = \mathcal{R}^\sharp(\alpha_1, \alpha_2).
  \end{align*}
\end{proof}

\begin{lemma}[Hodge-type decomposition]
  \label{lem:Hodge_flat}
  For any non-degenerate symmetric bilinear form $\mathcal{B}$ on $C_1(G,\R)$
  whose restriction to $Z=H_1(G,\R)\subset C_1(G,\R)$ is also
  non-degenerate, one has the direct sum decomposition
  \begin{equation}
    \label{eq:Hodge_flat}
    C^1(G,\R) = Z^\flat \oplus d\R^V.
  \end{equation}
\end{lemma}

\begin{remark}
  \label{rem:Hodge}
  This lemma shows that a Hodge-type decomposition may be arranged
  with respect to any non-degenerate symmetric bilinear form, not
  necessarily an inner product.  In fact, it can be shown that the sum
  is ``orthogonal'' with respect to the form, in the sense that
  \begin{equation}
    \label{eq:Hodge_orth}
    \mathcal{B}^\sharp(d\theta, \gamma^\flat) = 0
    \qquad
    \text{for any }
    \gamma^\flat\in Z^\flat, \theta \in \R^V.
  \end{equation}
\end{remark}

\begin{remark}
  \label{rem:Zflat_coords}
  Below we will set $\mathcal{B} = \mathcal{R}$, given by
  \eqref{eq:Rform_def}.  In coordinates, the space $Z^\flat$ is then
  given by $Z^\flat = R^{-1} Z$.
\end{remark}

\begin{proof}
  We start by dimension counting.  Recall that
  $C^1(G,\R) = \R^{\vec{E}}$, therefore $\dim C^1(G,\R) = |\E|$.  On
  the other hand, $v \mapsto v^\flat$ is an isomorphism (because the
  underlying form is non-degenerate), therefore
  $\dim Z^\flat = \dim Z = \beta = |\E| - |V| + c(G)$.  The dimension
  of $d\R^V$ is $\dim(d\R^V) = |V|-\dim\ker d = |V| - c(G)$, where
  $c(G)$ is the number of connected components of the graph $G$.
  Since the dimensions in \eqref{eq:Hodge_flat} agree, we only need to
  show that $Z^\flat \cap d\R^V = \{0\}$.

  Assume the contrary: $g^\flat = d \theta$ for some $\theta \in \R^V$
  and some cycle $g\in Z$.  Then, for an arbitrary cycle $\gamma \in
  Z$, \eqref{eq:flat_def} implies
  \begin{equation}
    \label{eq:no_intersection}
    \mathcal{B}(g, \gamma) = g^\flat \cdot \gamma = d\theta \cdot
    \gamma
    = \theta \cdot \partial \gamma = \theta \cdot 0 = 0.
  \end{equation}
  Since the form $\mathcal{B}$ was assumed to be non-degenerate even
  when restricted to $Z$, we conclude that $g=0$.
\end{proof}

The decomposition \eqref{eq:Hodge_flat} gives a \(\mathcal B\)-dependent
choice of representative of each cohomology class, namely an isomorphism
$H^1(G, \R) \simeq Z^\flat$.  When the form used is the form
$\mathcal{R}$ (whose restriction to cycles is the cycle intersection
form), we will call this isomorphism the \emph{Colin de Verdi\`ere
  gauge (CdV gauge)}, after its use in \cite{Col_apde13}.

\begin{lemma}
  \label{lem:gauge_useful}
  Let $\psi$ be an eigenvector of $H$ with non-zero entries and let
  $\mathcal{R}$ be the corresponding non-degenerate form on
  $C_1(G, \R)$.  Then for any $\alpha \in C^1(G,\R)$,
  \begin{equation}
    \label{eq:gauge_useful}
    \Psi \left.\frac{d}{dt} H_{t\alpha} \psi\right|_{t=0} = i \partial \alpha^\sharp,
  \end{equation}
  where $\Psi$ is the diagonal matrix with entries $\psi_v$ and
  $\partial$ is the boundary map \eqref{eq:boundary_map}.
  
  In particular, if $\gamma$ is a cycle ($\gamma \in Z$), then
  \begin{equation}
    \label{eq:gauge_useful0}
    \left.\frac{d}{dt} H_{t \gamma^\flat} \psi\right|_{t=0} =0.
  \end{equation}
\end{lemma}

\begin{remark}
  \label{rem:eigenvector_pert}
  From \eqref{eq:gauge_useful0} it follows that, for a simple
  eigenvalue, the corresponding normalized eigenvector branch may be
  chosen so that its first derivative vanishes for perturbations in
  the CdV gauge (tangent to \(Z^\flat\)).
\end{remark}

\begin{proof}
  We will establish the following ``weak'' formulation of
  \eqref{eq:gauge_useful},
  \begin{equation}
    \label{eq:weak_useful}
    \frac{d}{dt} \left<\xi, H_{t\alpha} \psi \right>
    = i \left(\Psi^{-1}\xi\right) \cdot \partial\alpha^\sharp,
    \qquad \text{for all }\xi \in \R^V,
  \end{equation}
  which is equivalent to \eqref{eq:gauge_useful} because $\Psi$ is
  invertible.  Here and below all derivatives are evaluated at $t=0$.
  Recalling that $r_e = -\left(\psi_u H_{u,v} \psi_v\right)^{-1}$, we compute
  \begin{equation*}
    \label{eq:ddt1}
    \frac{d}{dt} \left<\xi, H_{t\alpha} \psi \right>
    = \sum_{u\sim v}\xi_u \frac{d}{dt}\left(H_{u,v}
      e^{it\alpha_{u,v}}\right) \psi_v
    = i \sum_{u\sim v} \frac{\xi_u}{\psi_u}
    \big({\psi_u H_{u,v} \psi_v} \big) \alpha_{u,v}
    = -i \sum_{u\sim v} \frac{\xi_u}{\psi_u} \frac{1}{r_e} \alpha_{u,v}.
  \end{equation*}
  Grouping together the contributions from \((u,v)\) and \((v,u)\),
  for each \(e=(u,v)\in\vec E\), we obtain
  \begin{equation}
    \label{eq:ddt2}
    \frac{d}{dt} \left<\xi, H_{t\alpha} \psi \right>
    = i \sum_{e=(u,v)\in \vec{E}}
    \left(\frac{\xi_v}{\psi_v} - \frac{\xi_u}{\psi_u}\right)
    \frac1{r_e} \alpha_e
    = i \mathcal{R}^\sharp \left( d\left(\Psi^{-1} \xi\right),
      \alpha\right),
  \end{equation}
  where the sign difference arises due to $\alpha_{u,v}$ being
  antisymmetric.  Apply the definition of the dual form,
  \eqref{eq:dual_form_def_alt}, and the duality of $d$ and $\partial$
  to finish the proof:
  \begin{equation}
    \label{eq:ddt3}
    \mathcal{R}^\sharp \left( d\left(\Psi^{-1} \xi\right),
      \alpha\right)
    = d\left(\Psi^{-1} \xi\right) \cdot \alpha^\sharp
    = \left(\Psi^{-1} \xi\right) \cdot \partial \alpha^\sharp.
  \end{equation}
\end{proof}

\subsection{Proof of  \texorpdfstring{\Cref{thm:magHessian}}{the
    magnetic Hessian formulas}} 
\label{sec:mag_Hessian_proofs}

\begin{proof}[Proof of \Cref{thm:magHessian}]
  The computation proceeds --- and goes one step further --- along the
  standard path \cite{Col_apde13}, explained in detail in
  \cite[\S~5.4]{AloGor_jst23}.  We first point out that by classical
  perturbation theory results
  \cite{Kato_perturbation,Baumgartel_PertTheory}, the eigenvalue
  $\lambda_k$ and the corresponding eigenvector are analytic in the
  neighborhood of a point where the eigenvalue is simple.  We
  therefore have (at $\alpha=0$) the well defined linear functional
  $d\Lambda(0) : C^1(G,\R) \to \R$ and the bilinear form
  $\Hess \Lambda(0) : C^1(G,\R) \times C^1(G,\R) \to \R$,
  \begin{equation}
    \label{eq:dLambda_def}
    d\Lambda : \alpha \mapsto
    \left. \frac{\partial}{\partial t} 
      \Lambda(t\alpha) \right|_{t=0},
    \qquad
    \Hess\Lambda : (\alpha_1,\alpha_2)
    \mapsto
    \left. \frac{\partial^2}{\partial t_1 \partial t_2} 
      \Lambda(t_1\alpha_1+t_2\alpha_2) \right|_{t_1,t_2=0}.
  \end{equation}

  Since $H_\alpha$ is hermitian for all $\alpha$, we have
  \begin{equation}
    \label{eq:real_symm}
    \lambda_k\big(H_{-\alpha}\big) =
    \lambda_k\left(\overline{H_\alpha}\right) =
    \overline{\lambda_k\big(H_\alpha\big)} =
    \lambda_k\big(H_\alpha\big).  
  \end{equation}
  From this symmetry, we conclude that
  $d \Lambda(0) = 0$, i.e.\ $\alpha=0$ is a critical point.

  We now analyze the Hessian $\Hess \Lambda(0)$.  By \Cref{thm:main},
  equation~\eqref{eq:main_formula_nullity}, the cycle intersection
  form \(\mathcal X_r=\mathcal R|_Z\) is non-degenerate under the
  hypotheses of the theorem. Hence \Cref{lem:Hodge_flat} applies to
  \(\mathcal R\). Expanding $\alpha_i = \gamma_i^\flat + d\theta_i$
  according to the Hodge decomposition in \eqref{eq:Hodge_flat}, we
  note that
  \begin{equation}
    \label{eq:gauge_tt}
    \Lambda(t_1\alpha_1+t_2\alpha_2)
    = \Lambda(t_1\gamma_1^\flat+t_2\gamma_2^\flat + t_1d\theta_1 +
    t_2d\theta_2)
    = \Lambda(t_1\gamma_1^\flat+t_2\gamma_2^\flat),
  \end{equation}
  by gauge invariance.

  We now use the following formula for the second derivative of the
  simple eigenvalue $\Lambda(t_1,t_2)$ of a smooth hermitian matrix
  $H(t_1,t_2)$ (with the normalized eigenvector $\psi(t_1,t_2)$),
  \begin{equation}
    \label{eq:Hessian_formula}
    \frac{\partial^2 \Lambda}{\partial t_1 \partial t_2} 
    = \left< \psi, \frac{\partial^2 H}{\partial t_1 \partial t_2}
      \psi\right>
    + 2 \Real \left<\frac{\partial (H-\Lambda)}{\partial t_1} \psi,
      \frac{\partial\psi}{\partial t_2} \right>.
  \end{equation}
  Formula~\eqref{eq:Hessian_formula} can be obtained by
  differentiating the identity
  \begin{equation}
    \label{eq:QForm}
    0 = \left<\psi(t_1,t_2),\, \big(H(t_1,t_2)
      - \Lambda(t_1,t_2)\big) \psi(t_1,t_2) \right>,
  \end{equation}
  where \(\|\psi(t_1,t_2)\|=1\),
  with respect to $t_1$ and using
  $\big(H(t_1,t_2) - \Lambda(t_1,t_2)\big) \psi(t_1,t_2) = 0$ to
  obtain
  \begin{equation}
    \label{eq:HF}
    0 = \left<\psi(t_1,t_2),\, \frac{\partial}{\partial t_1}\Big(H(t_1,t_2)
      - \Lambda(t_1,t_2)\Big) \psi(t_1,t_2) \right>,
  \end{equation}
  and then differentiating again with respect to $t_2$.
  
  Since we already have
  $\frac{\partial \Lambda}{\partial t_1}(0,0) = 0$ from $d\Lambda=0$,
  in the second term of \eqref{eq:Hessian_formula} we evaluate
  \begin{equation}
    \label{eq:first_deriv}
    \frac{\partial}{\partial t_1}\left(H_{t_1 \gamma_1^\flat +
      t_2\gamma_2^\flat} \psi\right) \Big|_{t_1=t_2=0}
    = \left.\frac{d}{dt_1}
      \big(H_{t_1\gamma_1^\flat} \psi)\right|_{t_1=0} = 0 \in \C^V,
    \qquad \text{where } \psi=\psi(0,0),
  \end{equation}
  by \Cref{lem:gauge_useful}.  Therefore, we only need to compute the
  first term in \eqref{eq:Hessian_formula}, which by \Cref{lem:HessQF}
  evaluates to
  \begin{equation}
    \label{eq:Hessian_2}
    \frac12 \Hess \Lambda(0)
    = \mathcal{R}^\sharp \left(\gamma_1^\flat, \gamma_2^\flat\right)
      = \mathcal{R} \left(\gamma_1, \gamma_2 \right)
      = \mathcal{X}_r(\gamma_1, \gamma_2),
  \end{equation}
  where the last equality is true because $\mathcal{X}_r$ is the
  restriction $\mathcal{R}$ to cycles.

  We now observe that if
  $C^1(G,\R) \ni \alpha = \gamma^\flat + d\theta$ (where the flat is
  computed with respect to $\mathcal{R}$), then
  $[\alpha]^\sharp = \gamma \in H_1(G,\R)$ (where the sharp is
  computed with respect to the restricted form $\mathcal{X}_r$).
  Indeed,
  \begin{equation}
    \label{eq:def_sharp_check}
    \mathcal{X}_r([\alpha]^\sharp,g)
    = [\alpha] \cdot g
    = (\gamma^\flat + d\theta) \cdot g
    = \gamma^\flat \cdot g
    = \mathcal{X}_r(\gamma, g)
    \qquad \text{for all } g \in H_1,
  \end{equation}
  where we used the definition of sharp, \eqref{eq:sharp_def}, the
  definition of the pairing between $H^1$ and $H_1$, the fact that $d\theta
  \cdot g = 0$ for any $\theta \in \R^V$ and $g\in H_1$, and the
  definition of flat, \eqref{eq:flat_def}.  Since the form
  $\mathcal{X}_r$ is non-degenerate, comparing the start and
  end of \eqref{eq:def_sharp_check} we get $[\alpha]^\sharp = \gamma$.

  We now continue \eqref{eq:Hessian_2} with
  \begin{equation}
    \label{eq:Hessian_3}
     \mathcal{X}_r(\gamma_1, \gamma_2)
      = \mathcal{X}_r([\alpha_1]^\sharp, [\alpha_2]^\sharp)
      = \mathcal{X}_r^\sharp([\alpha_1], [\alpha_2]),    
  \end{equation}
  establishing formula~\eqref{eq:Hess}. 
  
  The coordinate representation follows from the observation that
  $\gamma$ in the decomposition $\alpha = \gamma^\flat + d\theta$ is
  given by the formula
  \begin{equation}
    \label{eq:alpha_to_gamma}
    \gamma = C\left(C^\top R C\right)^{-1} C^\top \alpha.
  \end{equation}
  Indeed, \(C^\top R C\) is invertible as a matrix representation of a
  non-degenerate form $\mathcal{X}_r$; the right-hand side of
  \eqref{eq:alpha_to_gamma}, being of the form $C\vec{x}$, is a cycle;
  and \(\gamma^\flat=R\gamma\) satisfies
  \(C^\top(\alpha-\gamma^\flat)=0\), which is the condition that
  \(\alpha-\gamma^\flat\) vanish on all cycles.  Using
  \eqref{eq:Hessian_2}, we get
  \begin{align*}
    \label{eq:Hessian_4}
    \frac12 \Hess \Lambda(0)
    = \mathcal{R} \left(\gamma_1, \gamma_2 \right)
    &= \left< C\left(C^\top R C\right)^{-1} C^\top \alpha_1,\,
    R C\left(C^\top R C\right)^{-1} C^\top \alpha_2 \right> \\
    &= \left<C^\top\alpha_1, \left(C^\top R C\right)^{-1} C^\top
      \alpha_2\right>,
  \end{align*}
  as claimed in equation \eqref{eq:Hess_coord}.
\end{proof}

\section{Application: dispersion relation of strained graphene}
\label{sec:graphene}
In this section we apply \Cref{thm:magHessian} to study the band
functions of the tight-binding model of strained graphene and its
variations, by computing the Hessian at the so-called corner or
symmetry critical points in the quasimomentum space.  The presentation
in this section follows ``physics style'', omitting proofs and using
our results somewhat beyond the context in which they were
established.

\subsection{The tight-binding model}

The tight-binding model for strained graphene
\cite{PerCasPer_prb09} is a graph whose vertex set $V$ is the union of
two lattices in $\R^2$, the $A$ lattice generated by arbitrary
linearly independent vectors $\vec{a}_1$ and $\vec{a}_2$, and the $B$
lattice obtained from $A$ by the shift by
$\vec\tau = (\vec{a}_1+\vec{a}_2)/3$, i.e.
\begin{equation}
  \label{eq:AB_def}
  A := \{n_1 \vec{a}_1 + n_2 \vec{a}_2 : n_1,n_2\in\Z\}
  \qquad
  B := \{n_1 \vec{a}_1 + n_2 \vec{a}_2 +\vec\tau : n_1,n_2\in\Z\}.
\end{equation}
The adjacency is defined by introducing an edge from each point
$\vec{x}$ in the $A$ lattice to points $\vec{x}+\vec\tau$,
$\vec{x}+\vec\tau-\vec{a}_1$ and $\vec{x}+\vec\tau-\vec{a}_2$ in the $B$
lattice, see \Cref{fig:graphene}.

\begin{figure}
\begin{tikzpicture}[scale=1.7]
    \coordinate (a1) at ({1}, 1/2);
    \coordinate (a2) at ({2/3}, -1/2);
    \coordinate (tauB) at ($ {1/3}*(a1) + {1/3}*(a2) $);

    \foreach \nOne in {2,1,0,-1} {
      \foreach \nTwo in {2,1,0,-1,-2} {

        \coordinate (RA) at ($\nOne*(a1) + \nTwo*(a2)$);
        \coordinate (RB) at ($(RA) + (tauB)$);

        \pgfmathtruncatemacro{\Sum}{\nOne + \nTwo}
        \ifthenelse{\Sum<3 \and \Sum > -3}{
            \draw[gray, very thick] (RA) -- (RB); 
          }{}
        \ifthenelse{\Sum < 4 \and \Sum > -3}{
          \fill[red] (RA) circle (0.08) node[below=1pt] {}; 
          \ifthenelse{\nOne > -1}{
            \draw[green, thick] (RA) -- ($(RA) - (a1) + (tauB)$); 
          }{}
          \ifthenelse{\nTwo > -2}{
            \draw[black, thick] (RA) -- ($(RA) - (a2) + (tauB)$); 
          }{}
          \fill[red] (RA) circle (0.08) node[below=1pt] {}; 
        }{}
        \ifthenelse{\Sum<3 \and \Sum > -4}{
          \fill[blue] (RB) circle (0.08) node[above=1pt] {}; 
        }{}
      }
    }
    \draw[->, dotted, thick, black] (0,0) -- (a1) node[midway, above] {$\vec{a}_1$};
    \draw[->, dotted, thick, black] (0,0) -- (a2) node[midway, below] {$\vec{a}_2$};
    \node[left, red] at (0,0) {$A$};
    \node[right, blue] at (tauB) {$B$};

    \path ( $(a1)+(a2)$ ) -- ( $(a1)+(a2)+(tauB)$ ) node[midway,below] {$t_3$};
    \path ( $(a2)+(tauB)$ ) -- ( $(a1)+(a2)$ ) node[midway,left] {$t_1$};
    \path ( $(a1)+(a2)$ ) -- ( $(a1)+(tauB)$ ) node[midway,right] {$t_2$};
  \end{tikzpicture}
  \begin{tikzpicture}[
    arr/.style={
      decoration={
        markings,
        mark=at position 0.65 with {\arrow{Stealth}}
      },
      postaction={decorate}
    },
    arr_rev/.style={
      decoration={
        markings,
        mark=at position 0.35 with {\arrowreversed{Stealth}}
      },
      postaction={decorate}
    }
    ]

    \coordinate (b) at (-2,0);
    \coordinate (mb) at (2,0);

    \draw[arr, bend left=50, draw=green,thick]
    (b) to node[draw=none, fill=white, above, midway] {$t_1$} (mb);
    \draw[arr, bend right=50,draw=black,thick]
    (b) to node[draw=none, fill=white, below, midway] {$t_2$} (mb);
    \draw[arr_rev,draw=gray,very thick]
    (b) to node[draw=none, fill=white, above, midway, yshift=0pt] {$t_3$} (mb);

    \fill[red] (b) circle (0.16) node[above,black,yshift=2pt] {$a$};
    \fill[blue] (mb) circle (0.16) node[above,black,yshift=2pt] {$b$};

  \end{tikzpicture}

  \caption{(Left) A finite piece of the strained graphene graph defined
    by the vectors $\vec a_1=(1,1/2)$ and $\vec a_2=(2/3, -1/2)$. (Right) The
    quotient of the infinite periodic graph by the period lattice
    $\Z^2$.}
  \label{fig:graphene}
\end{figure}
\begin{figure}
  \centering
  \begin{tikzpicture}[
    scale=1.7,
    Aatom/.style={red},
    Batom/.style={blue},
    Datom/.style={orange},
    tthree/.style={gray, very thick},
    tone/.style={green!60!black, thick},
    ttwo/.style={black, thick},
    tdec/.style={orange, thick}
]
\coordinate (a1) at (1, 1/2);
\coordinate (a2) at (2/3, -1/2);
\coordinate (tauB) at ($1/3*(a1)+1/3*(a2)$);
\coordinate (dvec) at (0,0.5);
\foreach \nOne in {-1,0,1} {
  \foreach \nTwo in {-1,0,1} {
    \pgfmathtruncatemacro{\Sum}{\nOne+\nTwo}
    \ifnum\Sum<2
      \ifnum\Sum>-2
        \coordinate (A0) at ($\nOne*(a1)+\nTwo*(a2)$);
        \coordinate (A1) at ($(A0)+(a1)$);
        \coordinate (A2) at ($(A0)+(a2)$);
        \coordinate (B0)  at ($(A0)+(tauB)$);
        \coordinate (D)  at ($(B0)+(dvec)$);
        \draw[tthree] (A0) -- (B0);
        \draw[tone]   (A1) -- (B0);
        \draw[ttwo]   (A2) -- (B0);
        \draw[tdec] (B0) -- (D);
      \fi
      \fi
  }
}
\draw[tone] ( $(a2)-(a1)$ ) -- ( $(a2)-2*(a1)+(tauB)$ );
\draw[ttwo] ( $(a1)-(a2)$ ) -- ( $(a1)-2*(a2)+(tauB)$ );
\draw[tthree] ( $(a2)+(a1)$ ) -- ( $(a2)+(a1)+(tauB)$ );
\foreach \nOne in {-1,0,1} {
  \foreach \nTwo in {-1,0,1} {
    \pgfmathtruncatemacro{\Sum}{\nOne+\nTwo}
    \coordinate (A0) at ($\nOne*(a1)+\nTwo*(a2)$);
    \coordinate (B)  at ($(A0)+(tauB)$);
    \coordinate (D)  at ($(B)+(dvec)$);

    \ifnum\Sum<3\ifnum\Sum>-1
    \fill[Aatom] (A0) circle (0.055);
    \fi\fi
    \ifnum\Sum>-2\ifnum\Sum<2
    \fill[Batom] (B) circle (0.055);
    \fi\fi
    \ifnum\nOne<1\ifnum\Sum>-1
    \fill[Datom] (D) circle (0.055);
    \fi\fi
  }
}
\node[left, red] at (0,0) {$A$};
\node[right, blue] at (tauB) {$B$};
\node[above right, orange] at ($(tauB)+(dvec)$) {$D$};
\path (0,0) -- (tauB)
  node[midway, below,yshift=1pt] {$t_3$};
\path (0,0) -- ($(tauB)-(a1)$)
  node[midway, right] {$t_1$};
\path (0,0) -- ($(tauB)-(a2)$)
  node[midway, right,xshift=-1pt] {$t_2$};
\path (tauB) -- ($(tauB)+(dvec)$)
  node[midway, left,xshift=1pt] {$t_d$};
\end{tikzpicture}
\begin{tikzpicture}[
    arr/.style={
      decoration={
        markings,
        mark=at position 0.65 with {\arrow{Stealth}}
      },
      postaction={decorate}
    },
    arr_rev/.style={
      decoration={
        markings,
        mark=at position 0.35 with {\arrowreversed{Stealth}}
      },
      postaction={decorate}
    }
    ]

    \coordinate (b) at (-2,0);
    \coordinate (mb) at (2,0);
    \coordinate (d) at (2,2);
    \draw[arr, bend left=50, draw=green,thick]
    (b) to node[draw=none, fill=white, above, midway] {$t_1$} (mb);
    \draw[arr, bend right=50,draw=black,thick]
    (b) to node[draw=none, fill=white, below, midway] {$t_2$} (mb);
    \draw[arr_rev,draw=gray,very thick]
    (b) to node[draw=none, fill=white, above, midway, yshift=0pt] {$t_3$} (mb);
    \draw[draw=orange,very thick]
    (mb) to node[draw=none,right, midway, yshift=0pt] {$t_d$} (d);

    \fill[red] (b) circle (0.16) node[above,black,yshift=2pt] {$a$};
    \fill[blue] (mb) circle (0.16) node[below,black,yshift=-2pt] {$b$};
    \fill[orange] (d) circle (0.16) node[above,black,yshift=2pt] {$d$};

  \end{tikzpicture}
  \caption{Strained decorated graphene and its quotient graph.}
  \label{fig:decorated_graphene}
\end{figure}

The graphene Hamiltonian $H$ is the self-adjoint operator on
$\ell^2(V)$ acting as
\begin{align}
  \label{eq:GrapheneHamA}
  (H \psi)_{a,n_1,n_2} &= a \psi_{a,n_1,n_2} + t_1
  \psi_{b,n_1-1,n_2} + t_2 \psi_{b,n_1,n_2-1}
  + t_3 \psi_{b,n_1,n_2} \\
  \label{eq:GrapheneHamB}
  (H \psi)_{b,n_1,n_2} &= b \psi_{b,n_1,n_2} + t_1
  \psi_{a,n_1+1,n_2} + t_2 \psi_{a,n_1,n_2+1}
  + t_3 \psi_{a,n_1,n_2},
\end{align}
where $a$, $b$, $t_1$, $t_2$ and $t_3$ are real constants and
$n_1,n_2\in \Z$ are arbitrary.  Note that we allow different
\emph{potentials} $a$ and $b$ at sites $A$ and $B$.  Physically, it
means that $B$ lattice may be composed of atoms different from the $A$
lattice; since graphene has identical atoms at every site, it
corresponds to the case $a=b$.

Because the Hamiltonian $H$ is periodic, it has the \emph{Bloch
  representation} as a direct integral of the $2\times2$ matrices,
cf.\ \Cref{sec:motivation_dispersion},
\begin{equation}
  \label{eq:H_Bloch}
  H(k_1, k_2) =
  \begin{pmatrix}
    a & t_1 e^{-ik_1} + t_2 e^{-ik_2} + t_3 \\
    t_1 e^{ik_1} + t_2 e^{ik_2} + t_3 & b
  \end{pmatrix}
\end{equation}
over the torus $(k_1,k_2)\in (\R/2\pi\Z)^2$.  Correspondingly, the
spectrum of $H$ is the union of the spectra of $H(k_1,k_2)$ over the
torus.  The eigenvalue functions \(\lambda_j=\lambda_j(k_1,k_2)\) are
known as the dispersion relations, or band functions; their union in
\((k_1,k_2,\lambda)\)-space is the real Bloch variety.  Relevant
questions include understanding the singularities and the extrema of
$\lambda_j(k_1,k_2)$, see \cite{HarKucSob_jpa07}, and whether
$\lambda_j$ is a perfect Morse function \cite{FauSot_prep25}.

We can also add periodic decorations to the graph, for example as
shown in \Cref{fig:decorated_graphene}.  The Bloch representation for
this graph becomes
\begin{equation}
  \label{eq:H_Bloch_deco}
  H(k_1, k_2) =
  \begin{pmatrix}
    a & t_1 e^{-ik_1} + t_2 e^{-ik_2} + t_3 & 0\\
    t_1 e^{ik_1} + t_2 e^{ik_2} + t_3 & b & t_d \\
    0 & t_d & d
  \end{pmatrix},
\end{equation}
which will be treated on the equal footing with \eqref{eq:H_Bloch}.

\subsection{The Hessian at the corner points}

We start by observing that the matrix $H(k_1,k_2)$ may be viewed as a
magnetic perturbation of a real Hamiltonian supported on the
\emph{quotient} graph shown in \Cref{fig:graphene}(right) or
\Cref{fig:decorated_graphene}(right), respectively.  We aim to
compute the Hessian of the eigenvalues of $H(k_1, k_2)$ at the
so-called \emph{corner points} $(0,0)$, $(0,\pi)$, $(\pi,0)$ and
$(\pi,\pi)$, which are critical due to symmetries analogous to
\eqref{eq:real_symm}.  The points other than $(0,0)$ will be handled
by changing the sign of the corresponding $t$ parameters; for example
$H(k_1,\pi+k_2)$ has the same form as \eqref{eq:H_Bloch} or
\eqref{eq:H_Bloch_deco} except for minus signs in front of $t_2$.

In the natural ordering of the edges (and the orientation
in \Cref{fig:graphene}), we take
\begin{equation}
  \label{eq:Cgraphene}
  C =
  \begin{pmatrix}
    1 & 0\\
    0 & 1\\
    1 & 1
  \end{pmatrix},
  \qquad
  \alpha = (k_1, k_2, 0),
\end{equation}
for the graph in \Cref{fig:graphene}(right). For the decorated graph,
the corresponding cycle frame is obtained by adding a zero row to
\(C\), corresponding to the decoration edge.  According to
\Cref{thm:magHessian} and \Cref{rem:fluxHessian}, the Hessian in the
flux coordinates $(k_1,k_2) = C^\top \alpha$ is the inverse of the
matrix
\begin{equation}
  \label{eq:Xmatrix}
  X := C^\top R C = -
  \begin{pmatrix}
    \dfrac{1}{\psi_a t_1 \psi_b} + \dfrac{1}{\psi_a t_3 \psi_b} &
    \dfrac{1}{\psi_a t_3 \psi_b} \\[10pt]
    \dfrac{1}{\psi_a t_3 \psi_b} &
    \dfrac{1}{\psi_a t_2 \psi_b} + \dfrac{1}{\psi_a t_3 \psi_b}
  \end{pmatrix},
\end{equation}
where $\psi_a$ and $\psi_b$ are the entries of the eigenvector at the
critical point.  Here we assume that the critical point eigenvalue is
simple and formally apply the formula for the Hessian to the quotient
graphs which are not simple, due to the parallel edges between \(a\)
and \(b\).

At a nondegenerate corner critical point, the nature of the critical
point is determined by the sign of $\det X$:
\begin{align*}
  \det X > 0 &\quad\Longleftrightarrow\quad \text{minimum or maximum},\\
  \det X < 0 &\quad\Longleftrightarrow\quad \text{saddle point}.
\end{align*}
A direct computation gives
\begin{equation}
  \label{eq:detX}
  \det X
  =
  \left(\frac{1}{\psi_a \psi_b}\right)^2
  \frac{t_1 + t_2 + t_3}{t_1 t_2 t_3}.
\end{equation}
Thus the sign is independent of the eigenvector entries, provided
\(\psi_a\psi_b\neq 0\).
Normalizing $t_3=1$, we want to understand the sign of $\det X$ at
four points $(t_1,t_2) = (\pm \tau_1, \pm \tau_2)$, where
$\tau_1,\tau_2> 0$.

The dependence of the sign of $\det X$ on $(t_1,t_2)$ is indicated by
shading in \Cref{fig:corner_types}.  Different values of $\tau_1$ and
$\tau_2$ lead to different counts of saddle points among the corner
critical points.  There are four possible configurations, shown in
\Cref{fig:corner_types}.  In one of the configurations,
\Cref{fig:corner_types}(left), three of the corner points are saddles
and only one is a min or a max.  In particular, in this case the band
functions can not be Morse perfect (cf. \cite{FauSot_prep25}).

\begin{figure}
  \centering
  \begin{tikzpicture}[>=Stealth]
    \fill[pink] (0,0) -- (3,0) -- (3,2.5) -- (0,2.5) -- cycle;
    \fill[pink] (-3,0) -- (-1,0) -- (-3,2) -- cycle;
    \fill[pink] (0,0) -- (-1,0) -- (0,-1) -- cycle;
    \fill[pink] (0,-1) -- (1.5,-2.5) -- (0,-2.5) -- cycle;
    \draw[thick,->] (-3,0) -- (3,0) node[right] {$t_1$};
    \draw[thick,->] (0,-2.5) -- (0,2.5) node[above] {$t_2$};
    \draw[thick] (-3,2) -- (1.5,-2.5) node[pos=0.9,right,xshift=2pt] {$t_1+t_2=-1$};;

    \def\px{1.3}
    \def\py{1}
    \draw[blue, dotted] (-\px,\py) -- (\px,\py) -- (\px,-\py) -- (-\px,-\py) -- cycle;

    \fill[blue]  (-\px, \py)  circle (2.5pt);
    \fill[red]   ( \px, \py)  circle (2.5pt);
    \fill[blue]  (-\px,-\py)  circle (2.5pt);
    \fill[blue]  ( \px,-\py)  circle (2.5pt);

    \node[red,above]  at (-\px, \py)  {$SP$};
    \node[red,above] at ( \px, \py)  {$M$};
    \node[red,left]  at (-\px,-\py)  {$SP$};
    \node[red,right] at ( \px,-\py)  {$SP$};
  \end{tikzpicture}
  \qquad
  \begin{tikzpicture}[>=Stealth]
    \fill[pink] (0,0) -- (3.2,0) -- (3.2,2.5) -- (0,2.5) -- cycle;
    \fill[pink] (-3.2,0) -- (-1,0) -- (-3.2,2.2) -- cycle;
    \fill[pink] (0,0) -- (-1,0) -- (0,-1) -- cycle;
    \fill[pink] (0,-1) -- (1.5,-2.5) -- (0,-2.5) -- cycle;

    \draw[thick,->] (-3.2,0) -- (3.2,0) node[right] {$t_1$};
    \draw[thick,->] (0,-2.5) -- (0,2.5) node[above] {$t_2$};
    \draw[thick] (-3.2,2.2) -- (1.5,-2.5);
    \def\px{0.7}
    \def\py{0.7}

    \draw[blue,dotted] (-\px,-3*\py) rectangle (\px,3*\py);
    \draw[blue,dotted] (-3.5*\px,-\py) rectangle (3.5*\px,\py);
    \draw[blue,dotted] (-\px/2,-\py/2) rectangle (\px/2,\py/2);

    \fill[blue]  (-\px, 3*\py) circle (2.5pt);
    \fill[red]  ( \px, 3*\py) circle (2.5pt);
    \fill[blue] (-\px,-3*\py) circle (2.5pt);
    \fill[red]  ( \px,-3*\py) circle (2.5pt);

    \fill[red]  (-3.5*\px,  \py) circle (2.5pt);
    \fill[red]  ( 3.5*\px,  \py) circle (2.5pt);
    \fill[blue] (-3.5*\px, -\py) circle (2.5pt);
    \fill[blue] ( 3.5*\px, -\py) circle (2.5pt);

    \fill[blue] (-\px/2,  \py/2) circle (2.5pt);
    \fill[red]  ( \px/2,  \py/2) circle (2.5pt);
    \fill[red] (-\px/2, -\py/2) circle (2.5pt);
    \fill[blue]  ( \px/2, -\py/2) circle (2.5pt);
  \end{tikzpicture}

  \caption{The sign of $\det X$ (shaded areas correspond to $\det
    X>0$) and the possible types for the quadruples of critical
    points. (Left) Three saddle points and one min/max. (Right) Two
    saddle points and 2 min/max points, which can happen in 3 possible
    ways.  The saddle points are indicated by blue circles and the
    min/max points by red circles.}
  \label{fig:corner_types}
\end{figure}

It is easy to conclude that the configuration in
\Cref{fig:corner_types}(left) is achieved if and only if
\begin{equation}
  \label{eq:triangle_ineq}
  \tau_1+\tau_2 > 1,
  \qquad
  1+\tau_1 > \tau_2,
  \qquad
  1+\tau_2 > \tau_1,
\end{equation}
or, in other words, the values $\{\tau_1,\tau_2,1\}$ satisfy the
triangle inequalities.  Breaking one of the conditions in
\eqref{eq:triangle_ineq} corresponds to one of the three
configurations in \Cref{fig:corner_types}(right).

\subsection{Interpretation: interior extrema and Dirac points}

Since an eigenvalue is continuous, it must achieve its other extreme
at a location other than a corner critical point.  In the $a=b$ case
of the undecorated graphene, the triangle inequalities are precisely
the ``gapless'' condition of Pereira, Castro Neto, and
Peres~\cite{PerCasPer_prb09}; it guarantees existence of Dirac points
(conical singularities) where the ``missing'' extremum is achieved
(see also \cite{BanBerWey_jmp15}).  Furthermore, the results of
\cite{AloBerGor_prep25} show that a \emph{smooth} critical point
$(k_1,k_2)$ outside the corner set can exist if and only if the same
condition holds: $\{t_1,t_2,t_3\}$ satisfy the triangle inequalities
(and then, the corresponding eigenvector of $H(k_1,k_2)$ vanishes
either at $a$ or at $b$).

To give an informal summary of the above discussion, the triangle
inequalities \eqref{eq:triangle_ineq} describe the setting in which an
extremum of $\lambda_j$ is achieved in the interior of the Brillouin
zone, either at a smooth critical point or at a conical point.
Surprisingly, this conclusion is entirely independent of the values of
the onsite potentials \(a,b\), and also independent of the presence of
the decoration and of its parameters.

\section{Application: Stability of oscillator networks}
\label{sec:oscillator_networks}

In this section, we apply \Cref{thm:main} to analyze stability of
stationary solutions to a coupled nonlinear systems of differential
equations.

There are many applications where it is useful to know where in the
spectrum some explicitly known eigenvector occurs. One example occurs
in the study of stability of exact solutions to nonlinear differential
equations. If, as is often the case, the equation admits some
symmetry, then the associated linearized operator will have a kernel,
with eigenvectors given by the action of the generators of the
symmetry on the solution. Let us assume that the kernel of the
linearized operator is one-dimensional, and that the operator itself
is Hermitian, as would be the case for a gradient flow. In this case
the position of $0$ in the spectrum indicates the dimension of the stable
and unstable manifolds. If the eigenvalues are given by $\lambda_1\leq
\lambda_2 \leq \ldots \lambda_{k-1} < 0 < \lambda_{k+1}\leq \ldots
\lambda_n$ then the linearized operator has inertia
$(n_-,n_0,n_+)=(k-1,1,n-k)$. In particular, the dimension of the
unstable manifold is $n-k$ and linearized stability is
only possible when $k=n$.
One well-studied example of nonlinear dynamics on a graph are the
so-called swing equations \cite{DorBul_siamjam12}
\begin{equation}
M_i\frac{d^2\theta_i}{dt^2} + \gamma_i \frac{d\theta_i}{dt} = \omega_i + \sum_{j:(ij)\in E} a_{ij} \sin(\theta_j-\theta_i-\phi_{ij}). \label{eqn:Swing}
\end{equation}
The swing equations model the synchronization of a network of
electrical generators, such as the Western interconnect, the
electrical grid covering most of the western United States and Canada.
The coefficients $M_i$ and $\gamma_i$ represent the inertial and
internal damping of the $i^{\text{\rm{th}}}$ generator in the network,
and $a_{ij}$ (assumed symmetric) represents the coupling strength
between generators $i$ and $j$. The nonzero $a_{ij}$ define a graph,
which typically has few cycles as compared with the number of
vertices/generators. The full swing equations are not a gradient flow,
but in many applications and in many analyses, the phase
shifts and inertial effects are neglected. This is equivalent to assuming that
$\phi_{ij}=0$ and $M_i/\gamma_i \ll 1$ respectively in equation
\eqref{eqn:Swing}. Under these assumptions it suffices to consider the
closely related non-uniform Kuramoto model \cite{Dor_Bul_IEEE}
\begin{equation}
  \label{eq:NUKM}
\gamma_i \frac{d\theta_i}{dt} = \omega_i + \sum_{j:(ij)\in E} a_{ij}\sin(\theta_j-\theta_i).  
\end{equation}
In the case where all of the generators have the
same internal damping $\gamma_i=1$, the linearization about a fixed point
($\frac{d\theta_i}{dt}=0$) denoted $\boldsymbol{\theta}^*$ is the
following Jacobian matrix supported on the graph
\[
  L_{ij} =
  \begin{cases} a_{ij} \cos(\theta^*_j-\theta^*_i), & i\neq j,\\
    -\sum_{k:(ik)\in E} a_{ik} \cos(\theta^*_k-\theta^*_i), & i=j.
  \end{cases}
\]
The number of positive eigenvalues of the Jacobian $L$ is the
dimension of the unstable manifold of the fixed point
$\boldsymbol{\theta}^*$ and hence determines the stability of
$\boldsymbol{\theta}^*$.  The matrix $L$ is of the form of the
combinatorial graph Laplacian with the important distinction that
classically the graph Laplacian is assumed to have all edge weights
positive: here the edge weight $a_{ij}\cos(\theta^*_j-\theta^*_i)$ can
take either sign, but will be assumed non-zero. For some discussion of
other contexts in which such $L$ arise see \cite{BroDev_siamjam14}.
The Kuramoto model is translation invariant,
$\theta_i \mapsto \theta_i+\alpha$, implying that the vector
$v = \mathbf{1} := (1,1,1,\ldots,1)^\top$ lies in the kernel of
$L$. This should also be clear from the definition of $L$, as the row
sum is zero.  In this context, the following result was derived to
count the dimension of the unstable manifold:
\begin{theorem}[Bronski, DeVille, Ferguson
  \cite{BroDevFer_siamjam16}]
  \label{thm:BDF}
  Assume the kernel of $L$ is spanned by the vector $\mathbf{1}$.  Let
  $D$ be a $|\E|\times|\E|$ diagonal matrix with entries
  $D_{i,j} = L_{i,j}$ and $C$ be a frame for the cycle space $Z$
  of the graph $G$.  Then
  \begin{equation}
    \label{eq:BDF}
    n_+(L) = n_-(D) - n_+\left(-C^\top D^{-1} C\right).
  \end{equation}
\end{theorem}
In the context of the present work, \Cref{thm:BDF} can be obtained
from \Cref{thm:main} applied to \(H=-L\), by setting \(\lambda=0\) and
\(\psi=\mathbf 1\), and translating between the Morse index of
\(\mathcal X_r\) and the inertia of \(C^\top D^{-1}C\).

In the general situation, when the parameters $\gamma_i>0$ in equation
\eqref{eq:NUKM} are different, the Jacobian matrix is $\Gamma^{-1}L$,
where $\Gamma$ is the diagonal matrix with $\Gamma_{ii} = \gamma_i$.
Thus the Jacobian is not symmetric but symmetrizable, being similar to
the $n\times n$ symmetric matrix
$H := \Gamma^{-\frac12} L \Gamma^{-\frac12}$.  The kernel contains the
vector $\mathbf{v}$ with $v_i = \sqrt{\gamma_i}$ generated by the
translational invariance, but since it is not $\mathbf{1}$,
\Cref{thm:BDF} cannot be applied directly and we need the more general
\Cref{thm:main} to determine the stability of $\boldsymbol{\theta}^*$.
Assuming this kernel is one-dimensional, the fixed point is stable
modulo translations precisely when
\begin{equation}
  \label{eq:StabCrit}
  n = \nodcnt(\mathbf{v},H) + 1 -  n_-(\mathcal X_r)
  \qquad \text{with } H = \Gamma^{-\frac12} L \Gamma^{-\frac12}
  \text{ and }\mathbf{v} = \left(\sqrt{\gamma_i}\right)_{i=1}^n,
\end{equation}
where the weights in \(\mathcal X_r\) are
\[
  r_e
  =
  -(v_i H_{i,j} v_j)^{-1}
  =
  -L_{i,j}^{-1},
  \qquad e=(i,j)\in\vec E.
\]
The difference between the left-hand side and right-hand side of
\eqref{eq:StabCrit} is the dimension of the unstable manifold.
\begin{example}
  \label{ex:kuramoto}
  Consider the Kuramoto model on a graph depicted in
  \Cref{fig:KGraph}, with $n=7$ and $\beta=8-7+1=2$.  The equations of
  motion are taken to be
  \[
    \gamma_i \frac{d\theta_i}{dt} = \omega_i + \sum_{j: (ij)\in E} \sin(\theta_j-\theta_i) \qquad \boldsymbol{\omega}=\begin{pmatrix}~~~0.05\\ -0.09 \\ ~~0.14 \\~~ 0.15 \\ -0.11 \\ -0.10 \\ -0.04 \end{pmatrix}  \qquad \boldsymbol{\gamma} = \begin{pmatrix}
      1 \\ 0.5 \\ 0.5 \\ 2\\ 2\\ 2 \\ 2  
    \end{pmatrix}.
  \]
  \begin{figure}
    \centering
    \includegraphics[width=0.4\linewidth]{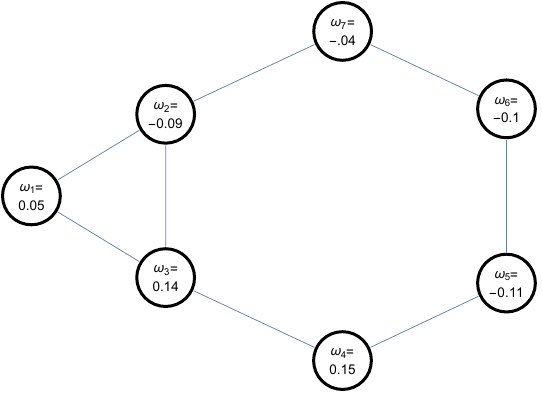}
    \caption{Kuramoto model on a graph with $\beta=2$}
    \label{fig:KGraph}
  \end{figure}

Such models will typically support multiple steady states whose
stability properties will vary. This particular example has
thirty-five fixed points, modulo translation, three of which are
depicted in figure \ref{fig:ClockGraphs}. We choose the translate so
that $\theta_1=0.$ The lines, or clock face, in each vertex indicate the angle of the
corresponding $\theta_i$, while the colors of the edges $(i,j)$
represent the sign of the corresponding entry of the symmetrized
stability matrix $H=\Gamma^{-\frac12} L \Gamma^{-\frac12},$ which is
given by the cosine of the angle difference, with positive entries
shown as black edges and negative entries in red.
\begin{figure}
    \centering
\includegraphics[width=0.3\linewidth]{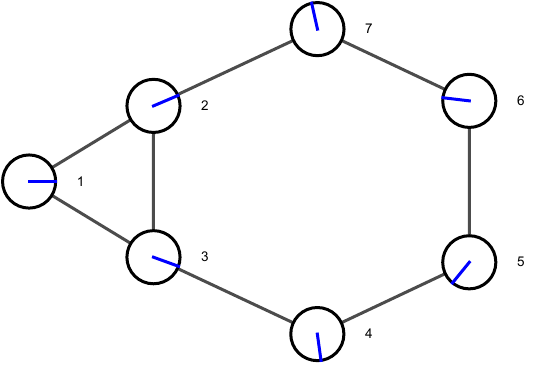}\includegraphics[width=0.3\linewidth]{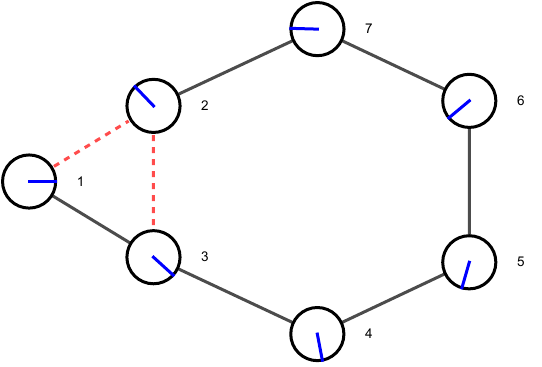}\includegraphics[width=0.3\linewidth]{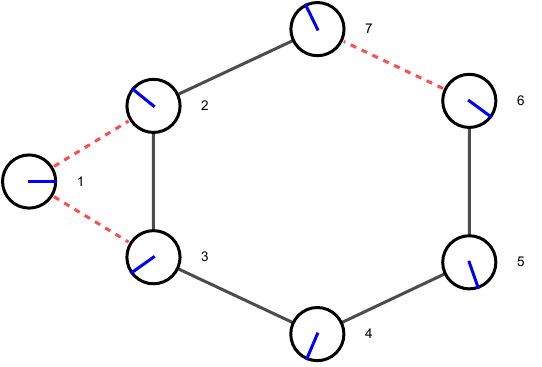}
    \caption{Three fixed points of the Kuramoto model defined on the
      graph in Figure \ref{fig:KGraph}. The ``clock face'' in each
      vertex represents the angle $\theta_i$. Red edges correspond to
      the negative entries of $H$, which occur when the angular
      difference between the vertices is greater than $\frac\pi{2}$.}
    \label{fig:ClockGraphs}
\end{figure}
We now check the stability criterion \eqref{eq:StabCrit} for the
linearized stability matrices
$H_{\boldsymbol{\theta}^*_{1,2,3}}:=\Gamma^{-\frac12}L_{\boldsymbol{\theta}_{1,2,3}^*}\Gamma^{-\frac12}$.
Since the eigenvector $\mathbf{v}$ has positive entries
$v_i = \sqrt{\gamma_i}$ we see that $\nodcnt(\mathbf{v},H)$ is simply
the number of edges $(i,j)$ such that $H_{i,j}>0$. This gives
$\nodcnt(\mathbf{v},H_{\boldsymbol{\theta}_1^*})=8$,
$\nodcnt(\mathbf{v},H_{\boldsymbol{\theta}_2^*})=6$ and
$\nodcnt(\mathbf{v},H_{\boldsymbol{\theta}_3^*})=5$.  For the last of
these it is worth noting that since $\nodcnt=5$, the maximum possible
value on the right-hand side of \eqref{eq:StabCrit} would be
$5+1=6<n$, so this solution is necessarily unstable, although the
exact dimension of the unstable manifold remains to be seen. In fact,
of the 35 steady states all but eight can be seen to be unstable by a
computation of the nodal count alone, without requiring the finer
information provided by the cycle intersection form.
A frame for $Z$ is
given by\footnote{This assumes a particular choice of orientation,
  which is hopefully clear.}
\[
  C^\top = \begin{pmatrix}
    1&1&1&0&0&0&0&0\\
    0&0&1&1&1&1&1&1
\end{pmatrix}
\]
and the cycle intersection form matrices are ($s=1,2,3$)
\begin{align*}
C^\top R_{\boldsymbol{\theta}_s^*} C = \begin{pmatrix}
    -3.52 & -1.37 \\
-1.37 & -15.39 \\
\end{pmatrix},\quad
\begin{pmatrix}
    1.10 & 1.00 \\
1.00 & -5.32 \\
\end{pmatrix}, \quad
\begin{pmatrix}
-1.15 & -3.66 \\
-3.66 & -7.38 \\
\end{pmatrix}.
\end{align*}
The first of these is negative definite and the latter two are
saddles. For the fixed point $\boldsymbol{\theta}_1^*$, we get
$\nodcnt(\mathbf{v},H) + 1 - n_-(\mathcal X_r) =8+1-2 = 7 = n$,
indicating that the fixed point is stable modulo the
symmetry-generated center manifold.  Similarly the zero eigenvalue for
the linearization about $\boldsymbol{\theta}_2^*$ satisfies
$\nodcnt(\mathbf{v},H) + 1 - n_-(\mathcal X_r) = 6+1-1=6$, implying a
$7-6=1$-dimensional unstable manifold.  Finally, the zero eigenvalue
of the linearization about the fixed point $\boldsymbol{\theta}_3^*$
satisfies $\nodcnt(\mathbf{v},H) + 1 - n_-(\mathcal X_r) =5+1-1=5$,
indicating a two-dimensional unstable manifold.
\end{example}

\section{Application: lower bound for the number of strong nodal domains}
\label{sec:nod_dom}

In this section we apply \Cref{thm:main} to provide a simple proof for
several existing lower bounds for the number of \emph{nodal domains}
of graph eigenvectors.
The study of nodal domains is fundamental in spectral theory. While
Courant's theorem provides an upper bound on both manifolds and
graphs\footnote{The list of references for the \emph{upper} bounds for
  nodal domains on graphs includes
  \cite{DavGlaLeySta_laa01,Fri_dmj93,FriLinSah_prep26,GlaZhu_qjmam02,Urs_laa18}.},
non-trivial lower bounds are only possible on graphs.

\begin{definition}
  \label{def:nod_dom}
  Let $H$ be strictly supported on $G$ with $H_{u,v}<0$ for $u\sim
  v$.  For an eigenvector $\psi$ of $H$, a \emph{strong nodal domain}
  is a maximal connected subgraph of $G$ induced by a set
  of vertices where $\psi$ has the same sign.  The \emph{strong nodal
    domain count}, $\nu(\psi)$, is the total number of such domains.
\end{definition}

Berkolaiko \cite{Ber_cmp08} established the first lower bound for the
nodal domain count on a general graph $G$,
\begin{equation}
  \label{eq:Bbound}
  \nu(\psi_k) \geq k - \beta,
\end{equation}
under the assumption that $\psi_k$ has no zero entries and the
corresponding eigenvalue $\lambda_k$ is simple.  H.~Xu and S.-T.~Yau
\cite{XuYau_jc12} extended his result by eliminating both of these
assumptions.  Their results were sharpened by Deidda, Putti and Tudisco
\cite{DeiPutTud_acha23} and, independently, by C.~Ge and S.~Liu
\cite{GeLiu_cvpde23}, assuming access to progressively more detailed
information about the eigenvector $\psi$.  The former paper proved it
in the more general setting of (nonlinear) $p$-Laplacians, while the
latter dropped the sign assumption on the off-diagonal entries of $H$.

Here we show how the Deidda--Putti--Tudisco bound follows almost
immediately from our \Cref{thm:main}, also relaxing the sign assumption
on the entries of $H$.

\begin{theorem}[Lower bound for the number of strong nodal domains,
  {\cite[Thm~3.10(P2)]{DeiPutTud_acha23}},
  {\cite[Thm.~6.6]{GeLiu_cvpde23}}]
  \label{thm:strong_nodal_bound}
  Let $\psi$ be an eigenvector of $H$ corresponding to an eigenvalue
  $\lambda$ with upper label $k^\uparrow$.  Let $z(\psi)$ be the
  number of vertices where $\psi$ is zero. Let $G'$ be the subgraph
  induced by the non-zero vertices of $\psi$, and let $\beta'$ be its
  first Betti number. Let $G_+$ be the subgraph of $G'$ obtained by
  removing the edges $(u,v)$ such that $\psi_u H_{u,v} \psi_v >0$ and
  let $\beta_+$ be its first Betti number. Then the strong nodal
  domain count $\nu(\psi)$, defined as the number of connected
  components of the graph $G_+$, satisfies
  \begin{equation}
    \label{eq:strong_nodal_bound}
    \nu(\psi) \ge k^\uparrow - z(\psi) + \beta_+ - \beta'.
  \end{equation}
\end{theorem}

\begin{remark}
  The new definition of the strong nodal domain count as $c(G_+)$
  agrees with \Cref{def:nod_dom} when $H_{u,v} < 0$ for $u\sim v$.
\end{remark}

\begin{proof}
  Observe that the edges removed from $G'$ to obtain $G_+$ are
  precisely those counted by $\nodcnt(\psi,H)$,
  equation~\eqref{eq:nodal_count_def}.  Denote by $V'$ the vertices
  where $\psi$ is non-zero.  Recalling the formula for the first Betti
  number, equation~\eqref{eq:beta_def}, we get (cf.~\cite[Eq.~(9)]{BerRazSmi_jpa12})
  \begin{align}
    \nu(\psi)=c(G_+)
    &= |V'|-|E(G_+)|+\beta_+ \notag\\
    &= |V'|-\big(|E(G')|-\nodcnt(\psi,H)\big)+\beta_+
    \label{eq:BRS}
    = c(G')-\beta' + \nodcnt(\psi,H)+\beta_+.
  \end{align}
  \Cref{thm:main} cannot be applied to compute $\nodcnt(\psi,H)$
  directly, due to the zero entries of $\psi$.  However, it can be
  applied to $\psi$ as an eigenvector of $H$ restricted to the vertices
  $V'$.  Denoting this restriction by $H'$, we have
  \begin{equation}
    \label{eq:ndcnt_prime}
    \nodcnt(\psi,H) = \nodcnt(\psi,H')
    \geq k^\uparrow(H') - c(G'),
  \end{equation}
  where we used \Cref{thm:main} and the trivial bound
  \(n_-(\mathcal X_r)\geq 0\).  Here \(k^\uparrow(H')\) denotes the
  upper label of the same eigenvalue \(\lambda\) for the restricted
  matrix \(H'\).  By the Cauchy interlacing theorem
  \cite[Thm~4.3.28]{HornJohnson},
  \begin{equation}
    \label{eq:cauchyZ}
    k^\uparrow(H') \geq k^\uparrow(H) - z(\psi).
  \end{equation}
  Combining \eqref{eq:BRS}, \eqref{eq:ndcnt_prime}, and
  \eqref{eq:cauchyZ} gives the desired result.
\end{proof}

\begin{remark}
  \label{rem:comparison}
  Since $\beta_+ \geq 0$ and $\beta' \leq \beta(G)$, inequality
  \eqref{eq:strong_nodal_bound} implies the bound $\nu(\psi) \ge k^\uparrow -
  z(\psi) - \beta$ of H.~Xu and S.-T.~Yau \cite[Thm~1.3]{XuYau_jc12}.
  
  Deidda, Putti and Tudisco, in \cite[Eq.~(5)]{DeiPutTud_acha23},
  provide a second version of the bound \eqref{eq:strong_nodal_bound},
  replacing the multiplicity of $\lambda$ (hidden here inside the
  definition of $k^\uparrow$) with $c(G')$, the number of connected
  components of $G'$.  Under some circumstances, this bound is better
  than \eqref{eq:strong_nodal_bound}.

  C.~Ge and S.~Liu \cite{GeLiu_cvpde23} improve the bound in
  \eqref{eq:strong_nodal_bound} by further analysing the connectivity of the
  vertices where $\psi$ is zero and thereby replacing $z(\psi)$ with a
  smaller quantity.
\end{remark}

\bibliography{bk_bibl}
\bibliographystyle{siam}

\end{document}